\PassOptionsToPackage{usenames,dvipsnames}{xcolor}
\documentclass[screen]{amsart}

\usepackage{hyperref}
\hypersetup{
    colorlinks=true,
    linkcolor=blue,
    filecolor=magenta,      
    urlcolor=blue,
    citecolor=blue,
    }

\usepackage{amsfonts, amsmath, latexsym, epsfig}
\usepackage{amssymb}
\usepackage{epsf}
\usepackage{url}
\usepackage{multirow,booktabs}
\usepackage{xcolor}
\usepackage{numprint}
\usepackage{siunitx}
\usepackage{cleveref}
\usepackage{algpseudocodex}
\usepackage{algorithm}

\algnewcommand{\LineComment}[1]{\State \(\triangleright\) #1}

\usepackage{tikz}
\usepackage{tikz-3dplot}
\usetikzlibrary{calc,trees,positioning,arrows,chains,shapes.geometric,%
    decorations.pathreplacing,decorations.pathmorphing,shapes,%
    matrix,shapes.symbols}

\newcommand{\RR}{\ensuremath{\mathbb{R}}}
\newcommand{\FF}{\ensuremath{\mathbb{F}}}

\newcommand{\QQ}{\ensuremath{\mathbb{Q}}}

\newcommand{\ZZ}{\ensuremath{\mathbb{Z}}}

\newtheorem{theorem}{Theorem}
\newtheorem{corollary}{Corollary}
\newtheorem{lemma}{Lemma}
\newtheorem{problem}{Problem}

\newtheorem{definition}{Definition}
\newtheorem{example}{Example}
\def\QuotS#1#2{\leavevmode\kern-.0em\raise.2ex\hbox{$#1$}\kern-.1em/\kern-.1em\lower.25ex\hbox{$#2$}}

\renewcommand{\S}{\ensuremath{\mathcal{S}}}
\renewcommand{\P}{\ensuremath{\mathcal{P}}}
\newcommand{\cG}{\ensuremath{\mathcal{G}}}

\newcommand{\lat}{\ensuremath{\mathcal{L}}}
\DeclareMathOperator{\Aut}{Aut}
\DeclareMathOperator{\Sym}{Sym}

\DeclareMathOperator{\GL}{GL}
\DeclareMathOperator{\cO}{\mathcal{O}}
\DeclareMathOperator{\cF}{\mathcal{F}}

\DeclareMathOperator{\spa}{Span}

\DeclareMathOperator{\Stab}{Stab}

\DeclareMathOperator{\Min}{Min}
\DeclareMathOperator{\Tr}{Tr}
\DeclareMathOperator{\Vol}{Vol}
\DeclareMathOperator{\Vor}{Vor}

\newcommand{\norm}[1]{\lVert #1 \rVert}
\newcommand{\ball}{\mathcal{B}}

\DeclareMathOperator{\Comb}{Comb}
\DeclareMathOperator{\Lin}{RLin}

\newcommand{\Alat}{\mathsf{A}}
\newcommand{\Dlat}{\mathsf{D}}
\newcommand{\Elat}{\mathsf{E}}
\newcommand{\Lambdalat}{\mathsf{\Lambda}}
\newcommand{\kappalat}{\mathsf{\kappa}}

\newcommand{\cpp}{{\tt C++}}
\newcommand{\cc}{{\tt C}}

\author{Mathieu Dutour Sikiri\'c}
\address{MSM programming, Karlova\v cka Cesta 28B, 10450 Jastrebarsko, Croatia}
\email{mathieu.dutour@gmail.com}

\author{Wessel van Woerden}
\address{PQShield, Amsterdam, The Netherlands}
\email{wessel.vanwoerden@pqshield.com}

\title{The lattice packing problem in dimension 9 by Voronoi's algorithm}
\date{}

\begin{document}

\begin{abstract}
In 1908, Voronoi introduced an algorithm that solves the lattice packing problem in any dimension in finite time. Voronoi showed that any lattice with optimal packing density must be a so-called perfect lattice, and his algorithm enumerates the finitely many perfect lattices up to similarity in a fixed dimension. However, due to the high complexity of the algorithm this enumeration had, until now, only been completed up to dimension 8.

In this work we compute all \numprint{2237251040} perfect lattices in dimension 9 via Voronoi's algorithm. As a corollary, this shows that the laminated lattice $\Lambdalat_9$ gives the densest lattice packing in dimension 9. Equivalently, we show that the Hermite constant $\gamma_9$ in dimension 9 equals $2$.
Furthermore, we extend a result by Watson (1971) and show that the set of possible kissing numbers in dimension 9 is precisely $2 \cdot \{ 1, \ldots, 91, 99, 120, \ldots, 129, 136 \}$.
\end{abstract}

\maketitle

\begin{center}
  \emph{In memory of Jacques Martinet (1939–2026).}
\end{center}

\section{Introduction}
The sphere packing problem is a classical and notoriously hard problem.
It asks how to pack equal $n$-dimensional balls in $\RR^n$ such that their density is maximized.
The sphere packing problem has only been solved in dimensions $1,2,3,8$ and $24$, the last two of which only recently~\cite{viazovska2017sphere,cohn2017sphere}.
Johannes Kepler conjectured already in the 17th century that in dimension $3$ no packing could be denser than the face-centered cubic or hexagonal close packings.
A testament of the hardness of this problem is given by the fact that Kepler's conjecture stayed unproven for almost $400$ years.
Until $1998$, when Thomas Hales announced a computational proof to the sphere packing problem in dimension $3$~\cite{hales2005kepler} which was later formalized with proof checking software in~\cite{hales2017formal}.

One obstacle that makes the sphere packing problem hard is that a priori there are no constraints on the configurations optimal sphere packings attain.
In particular it is not even clear that they form any regular pattern.
In contrast, the optimal sphere packings in dimension $1,2,3,8$ and $24$ are all very regular.
In fact, for these cases we can pick the centers of the balls to form a (discrete) additive group, i.e., a lattice.
A natural restriction of the problem is thus to look for optimal \emph{lattice packings}, also known as the lattice packing problem.

And indeed, this restriction allows to solve the problem in all dimensions $1$ to $8$ and $24$.
Contrary to Kepler's conjecture the lattice packing problem in dimension $3$ was already solved by Gauss in $1831$~\cite{gauss1831besprechung}.
Dimensions $4$ and $5$ followed in $1877$ by Korkine and Zolotarev's reduction theory~\cite{korkine1877formes}, and in $1935$ Blichfeldt used the same theory to determine the optimal lattice packings in dimensions $6, 7$ and $8$ \cite{blichfeldt1935minimum}.
The famous Leech lattice was shown to be optimal in dimension $24$ by Cohn and Kumar~\cite{cohn2009leechoptimality} using linear programming bounds, a technique which also played an important role in solving the general sphere packing problem in dimensions $8$ and $24$.
The optimality in all these dimensions was shown by strong theoretical arguments and limited computations.
However, the lattice packing problem in dimensions $9$ and beyond (except for $24$) so far seems to evade any of the theoretical approaches that worked for the other dimensions. In particular, all known density upper bounds are relatively far from the best known (lattice) packings in these dimensions.

Almost $100$ years before the computational proof of the Kepler conjecture, Voronoi, ahead of his time, already gave a computational approach to solve the lattice packing problem in any fixed dimension~\cite{VoronoiI}.
Voronoi showed that any optimal lattice packing must be a \emph{perfect lattice}, and that there are only a finite number of non-similar perfect lattices in any given dimension.
Furthermore, he presented an algorithm that provably enumerates all perfect lattices, thereby giving a computational way to solve the lattice packing problem.

\begin{table}
\centering
\caption{Number of non-similar perfect lattices.}\label{tab:nbforms}
  \begin{tabular}{rrl}\toprule
d & \# Perfect lattices & Authors \\ \midrule
2 & \numprint{1} & Lagrange, 1770~\cite{de1770demonstration} \\
3 & \numprint{1} & Gauss, 1831~\cite{gauss1831besprechung} \\
4 & \numprint{2} & Korkine \& Zolotarev, 1877~\cite{korkine1877formes} \\
5 & \numprint{3} & Korkine \& Zolotarev, 1877~\cite{korkine1877formes} \\
6 & \numprint{7} & Barnes, 1957~\cite{barnes1957perfect} \\
7 & \numprint{33} & Jaquet-Chiffelle, 1993~\cite{jaquet1993enumeration} \\
8 & \numprint{10916} & Dutour Sikiri\'c, Sch{\"u}rmann \&   \\
9 &  $\geq $ \numprint{500000} & Vallentin, 2005~\cite{PerfectDim8} \\
&  $\geq$ \numprint{23000000} & van Woerden, 2018~\cite{vanWoerdenMasterThesis} \\
&  \numprint{2237251040} & \textbf{this work} \\ \bottomrule
\end{tabular}
\end{table}

So far, this computational approach to solve the lattice packing problem, has always fallen behind the theoretical results, mainly due to its high complexity.
For example, the enumeration of perfect lattices in dimension $6$ was done by Barnes in $1957$~\cite{barnes1957perfect}, $22$ years after Blichfeldt's proof. And dimensions $7$ and $8$ followed, only in $1993$ by Jaquet~\cite{jaquet1993enumeration} and in $2007$ by Dutour Sikiri\'c, Sch{\"u}rmann and Vallentin~\cite{PerfectDim8} respectively.
% In short, lattice packing problem in dimension $n$ was already solved whenever the enumeration was done.
Dimension $9$ was also attempted, and in~\cite{PerfectDim8} and~\cite{AnzinEnumeration} about \numprint{500000} perfect lattices were found.
In~\cite{vanWoerdenMasterThesis} some additional algorithmic improvements lead to a total of \numprint{23000000} perfect lattices to be found.
Unfortunately, even with such a large number there was still no end in sight during the enumeration, and provably finding all cases requires a significant amount of extra computation, making it unclear if dimension $9$ would be feasible.

In this work we show that it is.
By many algorithmic, implementation and scalability improvements we present in this paper the full enumeration of all perfect lattices in dimension $9$, and as a by-product, we show that the laminated lattice $\Lambdalat_9$ is the densest lattice packing in dimension $9$.

\begin{theorem}\label{thm:perfect_and_extreme_forms}
  There are \numprint{2237251040} perfect and \numprint{7338582} extreme lattices in dimension $9$, of which the laminated lattice $\Lambdalat_9$ is the densest.
\end{theorem}

\begin{corollary}\label{cor:lambda9_is_densest}
The laminated lattice $\Lambdalat_9$ is the densest lattice packing in dimension $9$. The Hermite constant $\gamma_9$ equals $2$.
\end{corollary}
Given that $\gamma_8 = \gamma_9 = 2$ this also shows that the Hermite constant does not have to be strictly increasing in the dimension.
Furthermore, additional data obtained by our computation allows us to prove precisely how many shortest vectors a lattice in dimension $9$ can have.
\begin{theorem}\label{thm:poss_kissing_numbers}
The set of possible kissing numbers $|\Min(\lat)|$, for a lattice $\lat \subset \RR^9$ of dimension $9$, is $2 \cdot \{ 1, \ldots, 91, 99, 120, \ldots, 129, 136 \}$.
\end{theorem}
We want to emphasize that this has been a long-term project and that many of the technical improvements that were needed for this work have been presented and were used to attain other results in earlier works published by the authors and other co-authors. This culmination of works by us and others has led to the feasibility of the presented result. For sake of completeness we aim to give a high-level explanation of these improvements in this work.

\subsection*{Outline.}
In \cref{sec:preliminaries} we present some preliminaries on lattices, quadratic forms and polyhedra. Then in \cref{sec:ryshkov_s_polyhedra_and_perfect_forms,sec:voronoi_s_algorithm} we give a background on Voronoi's perfect form theory and a high-level description of Voronoi's algorithm.
In \cref{sec:canonical_forms} we treat several canonical functions that were vital to our computation.
In \cref{sec:the_dual_description}, which could be of independent interest, we explain the algorithm used to solve the dual description problem under symmetries, along with many algorithmic and implementation improvements.
Lastly, in \cref{sec:results} we discuss all results we obtained by successfully completing Voronoi's algorithm in dimension $9$.

\subsection*{Data and implementations}
We aim to make the data and implementations used in this work public. 
The full classification of perfect forms in dimension $9$ is available split into 64 binary \textbf{netCDF} files at~\cite{DutourSikiric2025Complete}. Their total size is roughly 150GB.
We wrote C++ and python utilities to easily access this data once downloaded\footnote{Available at \url{https://github.com/MathieuDutSik/AccessPerfectDim9} and \url{https://github.com/WvanWoerden/AccessPerfectDim9Py} respectively.}.
The dual description under symmetry program has been developed publicly at~\cite{PolyhedralCpp}.
Our parallel implementation of Voronoi's algorithm, including the canonical form implementation, will be made public in due course. 

\subsection*{Acknowledgements}
Significant parts of this work have been completed while W.~van Woerden was employed in the Cryptology group at CWI in Amsterdam and in the Algorithmic Number Theory group at the University of Bordeaux.
A small part of the computations presented in this work were carried out on the Lisa Compute Cluster, a component of the Dutch national e-infrastructure, which was operated by SURF. Most of the computations were performed on the Curta cluster of the Mésocentre de Calcul Intensif Aquitain (MCIA).

The authors would like to thank Achill Sch\"urmann and Jacques Martinet for insightful discussions and suggestions that were beneficial to this work.
The authors would also like to thank Matthias Schymura and Gregory Minton for useful feedback on the first draft of this work.

\section{Preliminaries}
\label{sec:preliminaries}

\subsection{Lattices and lattice packings}

\begin{definition}[Lattice]
  For a basis $B \in \RR^{d \times d}$ with linearly independent columns $b_1, \ldots, b_d \in \RR^d$ we define the \emph{lattice} generated by $B$ as
  $$\lat(B) := B \cdot \ZZ^d = \left\{ \sum_{i=1}^d x_i d_i : (x_1, \ldots, x_d) \in \ZZ^d \right\}.$$
  For a lattice $\lat$ with basis $B$ we denote $\det(\lat) := |\det(B)|$ and $\lambda_1(\lat) := \min_{v \in \lat \setminus \{0\}} \norm{v}$ for its (co)volume and first minimum respectively.
\end{definition}
A lattice basis is not unique, in particular for any unimodular transformation $U \in \GL_d(\ZZ)$ the bases $B$ and $BU$
generate the same lattices.
Furthermore, we call two lattices $\lat, \lat'$ isomorphic if $\lat' = O \cdot \lat := \{ Ov : v \in \lat \}$ for an orthonormal
transformation $O \in \cO_d(\RR)$.
Such an isomorphism, if it exists, is unique up to composition with an \emph{automorphism} $O' \in \Aut(\lat) := \{ O' \in \cO_d(\RR) : O' \cdot \lat = \lat \}$. If two lattices are isomorphic up to some positive scaling we call them \emph{similar}.
The space of lattices up to isomorphisms can be identified with the double coset $\cO_d(\RR) \setminus \GL_d(\RR) / \GL_d(\ZZ)$.

The first minimum $\lambda_1(\lat)$ of a lattice denotes the minimum length of any nonzero vector, and we denote the set of \emph{minimal vectors} attaining this length by $\Min(\lat) := \{ v \in \lat : \norm{v} = \lambda_1(\lat) \}$.
Equivalently, $\lambda_1(\lat)$ is the minimum distance between any two distinct lattice points, and thus any lattice $\lat$ can be turned into a sphere packing by placing balls of radius $\frac{1}{2}\lambda_1(\lat)$ around each lattice point.
Such packings are called \emph{lattice packings} or \emph{regular sphere packings}, and they attain the following packing density.
\begin{definition}[Packing density]
For a lattice $\lat \subset \RR^d$ of dimension $d \geq 1$ we define its packing density $\delta(\lat)$ by
$$\delta(\lat) := \frac{\lambda_1(\lat)^d \cdot \Vol(\ball^d)}{2^d \cdot \det(\lat)} \sim \frac{\lambda_1(\lat)^d}{\det(\lat)}.$$
\end{definition}
The packing density $\delta(\lat)$ corresponds to the density of the corresponding lattice packing as is shown in~\cref{fig:latticedensity}.
Note that the volume, first minimum, and thus the packing density of a lattice are invariant under orthonormal transformations. Isomorphic lattices therefore have the same packing density. Moreover, as scaling does not change the density, similar lattices also have the same packing density.
A natural question then follows: what are the densest lattice packings up to similarity?
\begin{problem}[Lattice Packing Problem]
For a dimension $d \geq 1$, which lattices $\lat \subset \RR^d$ (up to similarity) attain the largest packing density?
\end{problem}

\begin{example}
As an example let use consider the best lattice packing $\Elat_8$ in dimension $8$ and the (until this work) conjectured best lattice packing $\Lambdalat_9$ in dimension $9$. The $\Elat_8$ root lattice has a relative simple description as follows
$$\Elat_8 := \{ x \in \ZZ^8 \cup (\ZZ+\tfrac{1}{2})^8 : \sum_{i=1}^8 x_i \equiv 0 \bmod{2}  \}.$$
It has a determinant of $1$, a first minimum of $\sqrt{2}$ and $240$ minimal vectors, leading to a density of $\pi^4/(2^4 \cdot 4!) \approx 0.25367$ which was shown to be optimal by Blichfeldt~\cite{blichfeldt1935minimum}.
The laminated lattice $\Lambdalat_9$ can be constructed by stacking the $\Elat_8$ root lattice. The unit vector $e_1 = (1,0,0,0,0,0,0,0)$ is a deep hole of $\Elat_8$ on which we can stack another copy of $\Elat_8$ by using the gluing vector $e_1 + e_9$ of length $\sqrt{2}$. We then obtain (one version of) the laminated lattice as follows
$$\Lambdalat_9 := \{ (x~|~0) \in \RR^9~|~x \in \Elat_8  \} + (e_1+e_9) \cdot \ZZ.$$
This lattice has determinant $1$, a first minimum of $\sqrt{2}$ and $272$ minimal vectors, leading to a density of $(4! \cdot 2^{9/2} \cdot \pi^4)/9! \approx 0.14577$. Typically we scale $\Lambdalat_9$ by $\sqrt{2}$ to make it integral.
\end{example}

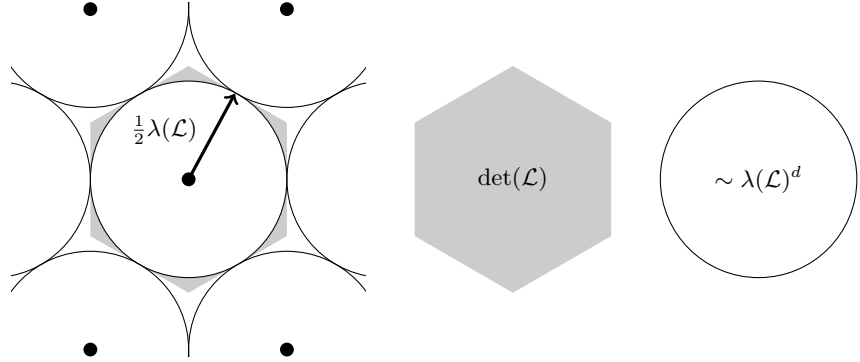
\begin{figure}
\begin{tikzpicture}[scale=1.3]

    \begin{scope}
    \clip (-1.80,-1.80) rectangle (1.80,1.80);
  \fill[opacity=0.2] (1,0.57735027) -- (0,1.15470054) -- (-1,0.57735027) -- (-1,-0.57735027) -- (0,-1.15470054) -- (1,-0.57735027) -- cycle;
  \fill[white] (0,0) circle (1);
  \draw (0,0) circle (1);
  \fill (0,0) circle[radius=2pt];
  \draw (1,1.73205081) circle (1);
  \fill (1,1.73205081) circle[radius=2pt];
  \draw (-1,1.73205081) circle (1);
  \fill (-1,1.73205081) circle[radius=2pt];
  \draw(-1,-1.73205081) circle (1);
  \fill (-1,-1.73205081) circle[radius=2pt];
  \draw (1,-1.73205081) circle (1);
  \fill (1,-1.73205081) circle[radius=2pt];
  \draw (2,0) circle (1);
  \fill (2,0) circle[radius=2pt];
  \draw (-2,0) circle (1);
  \fill (-2,0) circle[radius=2pt];
  \draw[very thick,->] (0,0) -> (0.47,0.87);
  \draw (-0.25,0.5) node {\small$\frac{1}{2}\lambda(\lat)$};
  \end{scope}

 \begin{scope}[shift={(3.3,0)}]
  \fill[opacity=0.2] (1,0.57735027) -- (0,1.15470054) -- (-1,0.57735027) -- (-1,-0.57735027) -- (0,-1.15470054) -- (1,-0.57735027) -- cycle;
  \draw (0,0) node {\small $\det(\lat)$};
  \fill[white] (2.5,0) circle (1);
  \draw (2.5,0) circle (1);
  \draw (2.5,0) node {\small $\sim \lambda(\lat)^{d}$};
  \end{scope}
\end{tikzpicture}
\caption{The density of a lattice packing is determined by its first minimum $\lambda_1(\lat)$ and its (co)volume $\det(\lat)$.}\label{fig:latticedensity}
\end{figure}

\subsection{Quadratic forms and the Hermite constant}
\label{sub:quadratic_forms_and_the_hermite_constant}
The literature on lattice packings is usually phrased in the setting of quadratic forms.
We will follow this convention in this work and shortly explain here the relation between lattices and (positive definite) quadratic forms.
This can be seen as interpreting the space of lattices up to isomorphisms $\cO_d(\RR) \setminus \GL_d(\RR) / \GL_d(\ZZ)$ in a different way.

With a matrix $Q \in \RR^{d \times d}$ we can associate a \emph{quadratic form} $\RR^d \to \RR, x \mapsto Q[x] := x^t Q x$ in $d$ variables. Note that we can assume without loss of generality that $Q$ is symmetric and thus we identify the space of quadratic forms in $d$ variables with the space $\S^d \subset \RR^{d \times d}$ of real symmetric matrices of dimension $n := \frac{1}{2}d(d+1)$.
We can furthermore define the Frobenius inner product on $\S^d$:
    \begin{align*}
      \langle A,B \rangle := \Tr(A^tB) = \sum_{i,j} A_{ij}B_{ij}.
    \end{align*}
Note that for a lattice with basis $B \in \GL_d(\RR)$ the \emph{gram matrix} $Q := B^tB$ is a symmetric matrix and thus a quadratic form.
Furthermore, the gram matrix $Q$ is \emph{positive definite} leading to a \emph{positive definite quadratic form} (PQF),
which induces an inner product $\langle x,y \rangle_Q := x^tQy = \langle Bx, By \rangle$ and a norm $\norm{x}_Q := \sqrt{Q[x]} = \norm{Bx}$ on $\RR^d$.
In the other direction, given a PQF $Q$, one can compute the Cholesky decomposition $Q = C^t C$, giving a lattice $\lat(C)$ with gram matrix $Q$. Note that this is not unique, in particular any basis $O \cdot C$ for $O \in \cO_d(\RR)$ has the same gram matrix $Q$. The set of PQF's $\S_{>0}^d$ can thus be identified with the left coset $\cO_d(\RR) \setminus \GL_d(\RR)$. We call two PQF's $Q, Q' \in \S_{>0}^d$ \emph{equivalent} if $Q' = U^t Q U$ for some unimodular transformation $U \in \GL_d(\ZZ)$, and \emph{similar} if they are equivalent up to scaling by some $\alpha > 0$. This precisely corresponds to lattice isomorphism and similarity respectively. Just as in the lattice case, such a transformation is unique up to composition with an \emph{automorphism} $V \in \Aut(Q) := \{ V \in \GL_d(\ZZ) : V^tQV = Q \}$. Note that the space of PQF's up to equivalence can thus also be seen as the double coset $\cO_d(\RR) \setminus \GL_d(\RR) / \GL_d(\ZZ)$. Lattices and PQF's are simply different interpretations of the same object.
For a PQF $Q \in \S_{>0}^d$ we can therefore also define similar properties as for a lattice, namely its determinant $\det(Q)$, and its first minimum and minimal vectors
  \begin{align*}
    \lambda(Q) := \min\limits_{x \in \ZZ^d - \{ 0 \}} Q \lbrack x \rbrack, \quad & \quad \Min(Q) := \{ x \in \ZZ^d : Q\lbrack x \rbrack = \lambda(Q) \}.
  \end{align*}
For a lattice $\lat$ with basis $B$ and gram matrix $Q = B^t B$ we have the following correspondence $\det(Q) = \det(\lat)^2$, $\lambda_1(Q) = \lambda_1(\lat)^2$, and $\Min(Q) = \{ B^{-1}v : v \in \Min(\lat) \}$. In particular a lattice vector $v = Bx \in \lat(B)$ corresponds to the vector $x \in \ZZ^d$ w.r.t. the PQF $Q = B^t B$.

In the setting of PQF's one does usually not consider the packing density, but the highly related \emph{Hermite invariant}.
\begin{definition}[Hermite invariant and constant]
For a PQF $Q \in \S_{>0}^d$ of dimension $d \geq 1$, we define its Hermite invariant by
    $$\gamma(Q) := \frac{\lambda_1(Q)}{\det(Q)^{1/d}}.$$
For any dimension $d \geq 1$ the \emph{Hermite constant} $\gamma_d$ is defined as
    $$\gamma_d := \sup_{Q \in \S_{>0}^d} \gamma(Q).$$
\end{definition}
The supremum in the definition of the Hermite constant can in fact be replaced by a maximum as it is attained by some form (for example this follows from the perfect form theory in~\cref{sec:ryshkov_s_polyhedra_and_perfect_forms}).
For a fixed dimension $d \geq 1$ and a lattice $\lat \subset \RR^d$ with gram matrix $Q \in \S_{>0}^d$ we have
\begin{align*}
  \delta(\lat) \sim \frac{\lambda_1(\lat)^d}{\det(\lat)} = \left( \frac{\lambda_1(Q)}{\det(Q)} \right)^{d/2} = \gamma(Q)^{d/2},
\end{align*}
and thus the packing density of the lattice $\lat$ is directly related to the Hermite invariant of the PQF $Q$, and the optimal packing density is directly related to the Hermite constant. Two forms that are equivalent or even only similar naturally have the same Hermite invariant. The lattice packing problem in dimension $d \geq 1$ can thus be rephrased as follows: which PQF $Q \in \S_{>0}^d$ (up to similarity) maximizes the Hermite invariant $\gamma(Q)$?

\subsection{Polytopes}
\label{sub:polytopes}
Consider the Euclidean vector space $\RR^n$ with the standard inner product $\langle x,y \rangle = x^\top y$.
For any vector $y \in \RR^n$ and real number $c \in \RR$ we can denote the (affine) \emph{hyperplane} by $H_{y,=c} = \{ x \in \RR^n : \langle x,y \rangle = c \}$ and the (affine) \emph{halfspace} by $H_{y,\geq c} = \{ x \in \RR^n : \langle x,y \rangle \geq c \}$. We define a (convex) \emph{polyhedron} $P = \bigcap_{i=1}^k H_{y_i, \geq c_i}$ as an intersection of a finite number of halfspaces. We call a polyhedron that is bounded a \emph{polytope}.

By grouping the vectors $y_1, \ldots, y_k$ together as the rows of a matrix $A$, and the constants $c = (c_1, \ldots, c_k)^t$ we can also describe a polyhedron by $P = \{ x \in \RR^n : Ax \geq c \}$, where the inequality $Ax \geq c$ has to be satisfied coefficient-wise.

If $c = (0, \ldots, 0)^t$ we call $P = \{ x \in \RR^n : Ax \geq 0\}$ a \emph{polyhedral cone}. For a polyhedral cone $P$, $x \in P$ and a positive scalar $\lambda \geq 0$ we always have $\lambda x \in P$. A polyhedral cone $P$ is called \emph{pointed} if $A$ has full-rank, or equivalently, if no nontrivial linear subspace is contained in it.

For an affine halfspace $H_{y,\geq c}$ that fully contains a polyhedron $P$, we call any intersection $F = P \cap H_{y,=c}$ a \emph{face} of $P$ and $\langle y, x \rangle \geq c$ for $x \in \RR^n$ a \emph{face-defining inequality} of $F$. We call the affine dimension of $F$ the dimension of the face denoted by $\dim(F)$. A face of dimension $0$ or $\dim(P)-1$ is called a \emph{vertex} or \emph{facet} respectively. A bounded face of dimension $1$ is called a \emph{ridge}. A pointed polyhedral cone $P$ has $0$ as the only vertex, and we call its (unbounded) faces of dimension $1$ \emph{extreme rays}.

For simplicity we restrict the following to \emph{pointed polyhedral cones} and we assume without loss of generality that the cone $P \subset \RR^n$ has full dimension $\dim(P) = n$.
We have already seen that (by definition) a polyhedral cone $P$ can be described by an intersection of a finite number of half-spaces given by vectors $a_1, \ldots, a_k$.
In fact, there exists a unique minimal set of such half-spaces, precisely corresponding to the facets of $P$.
We call this the \emph{facet or half-space representation} of $P$, often referred to as the \emph{H-representation} and typically represent it by $H(P) = A \in \RR^{k \times n}$ where $A$ has a minimal number of rows $a_1, \ldots, a_k$ such that $P = \{ x \in \RR^n : Ax \geq 0 \}$.

Now let $v_1, \ldots, v_m$ be some representative vectors for the extreme rays $\{ \lambda_1 v_1 : \lambda_1 \geq 0 \}, \ldots, \{ \lambda_m v_m : \lambda_m \geq 0 \}$ of $P$.
Then equivalently, one can describe $P$ as a convex combination of the vertex $0$ and its extreme rays, i.e., as $P = \{ \sum_{i=1}^m \lambda_i v_i : \lambda_i \geq 0 \}$. This is also called the extreme ray or \emph{vertex representation} of $P$, or equivalently its \emph{V-representation}. We typically denote $V(P) = V \in \RR^{n \times m}$ with columns $v_1, \ldots, v_m$ such that $P = \{ V\lambda \in \RR^n : \lambda \in \RR_{\geq 0}^m \}$.
We denote $|H(P)|$ and $|V(P)|$ for the number of facets and extreme rays of $P$ respectively.

The V-representation and H-representations are in some sense dual to each-other. More precisely, given a full-dimensional cone $P$ one can define its dual cone $$P^\circ = \{ y \in \RR^n : \langle x,y \rangle \geq 0~\forall x \in P \}.$$
Then $A \in \RR^{k \times n}$ is an H-representation of $P$ if and only if $A^t$ is a V-representation of $P^\circ$ and vice versa.
I.e., there is a one to one correspondence between the facets of $P$ and the extreme rays of its dual $P^\circ$.
Because $(P^\circ)^\circ = P$ there is a similar correspondence between the extreme rays of $P$ and the facets of $P^\circ$.

The problem of computing one representation, given the other, is known as the \emph{dual description problem},
and this is one of the main problems that we have to solve in this work.
For us it is sufficient to consider \emph{rational} cones that can be represented by rational inequalities or rational generators of the extreme rays.
\begin{problem}[Dual Description Problem]
  % Let $P = \{ x \in \RR^n : Ax \geq 0 \}$ be a pointed polyhedral cone given by an H-representation $A \in \QQ^{m \times n}$. Recover a V-representation of $P$.
  Let $P = \{ V \lambda \in \RR^n : \lambda \in \RR_{\geq 0}^m \}$ be a pointed polyhedral cone given by a V-representation $V \in \QQ^{n \times m}$. Recover an H-representation of $P$.
\end{problem}
By considering the dual $P^\circ$ instead of $P$, this problem is identical to the recovery of an H-representation from a V-representation.
The main obstacle for the dual description problem is that the H- and V-representations can differ wildly in size. For example, we will encounter a polyhedral cone $P \subset \RR^{45}$, related to the laminated lattice $\Lambdalat_9$, with only \numprint{136} extreme rays, but with \numprint{5221782341716704} facets.
Clearly, computing all these facets is beyond feasible.

This is where polyhedral symmetries, or automorphisms, come in.
The set of faces of a polyhedron $P$ form a poset under face inclusion. Moreover, for each pair of faces $F,G \subset P$, we have that the \emph{meet} $F \cap G$ is a face of $P$, and there exists a unique minimal face called the \emph{join} that contains $F \cup G$.
This implies that the poset of faces is in fact a \emph{poset lattice}, also called the \emph{face lattice}, not to be confused with the Euclidean lattices that form the main object of interest in this work.
The face lattice is both atomic and co-atomic, meaning that every face $F$ can uniquely be described by the facets it is contained in or by the extreme rays it contains.

The {\em combinatorial automorphism group} $\Comb(P)$ of a polyhedron $P$ is the group of permutations of the face lattice, that preserves inclusion relations.
By the atomic or co-atomic property these automorphisms are fully determined by their action on the set of facets or extreme rays, and thus we can view $\Comb(P)$ either as a subgroup of the permutation groups $\Sym(|V(P)|)$ or $\Sym(|H(P)|)$, depending on the situation.

More precisely, if a polyhedron has the extreme rays $\{\RR v_i\}_{1\leq i\leq m}$ then a face $F$ is uniquely
described by the set of extreme rays it contains.
One can therefore represent each face $F$ by a subset $S_F \subset \{ 1, \ldots, m \}$, and we can interpret $\Comb(P)$ as a subgroup of $\Sym(m)$, which acts on a face represented by $S_F$ as $\sigma \circ S_F = \{ \sigma(i) : i \in S_F \}$. We call this the \emph{incidence representation}. Following this naming we call the number of extreme rays $|S_F|$ the \emph{incidence} of $F$.

The combinatorial automorphism group is generally difficult to compute since the faces are precisely what interests us.
Again considering the extreme rays $\{\RR v_i\}_{1\leq i\leq m}$ for some representative vectors $v_1, \ldots, v_m$, the {\em restricted linear automorphism group} $\Lin_{.}(P)$ is the group of permutations $\sigma \in \Sym(n)$ such that there exists a matrix $B \in \GL_n(\RR)$ for which $Bv_i = v_{\sigma(i)}$ for all $1\leq i \leq m$.
Note that the obtained group can depend on the precise representatives $v_i$.
However, we do always have that $\Lin_{.}(P) \subset \Comb(P)$.
Furthermore, this group is efficiently computable in practice by reducing it to a graph automorphism problem. We refer to \cite{GroupPolytopeLMS} for more theoretical and computational details.
For our purposes it is sufficient to only obtain the facet orbits up to some symmetry group $G \subset \Comb(P)$.
\begin{problem}[Dual Description Problem under Symmetry]
  % Let $P = \{ x \in \RR^n : Ax \geq 0 \}$ be a pointed polyhedral cone given by an H-representation $A \in \QQ^{m \times n}$ and a subgroup $G \subset \Comb(P) \cong \Sym(m)$ of combinatorial automorphisms. Recover a V-representation of $P$ up to the action of $G$, i.e., recover a single representative of each orbit.
  Let $P = \{ V \lambda \in \RR^n : \lambda \in \RR_{\geq 0}^m \}$ be a pointed polyhedral cone given by a V-representation $V \in \QQ^{n \times m}$ and a subgroup $G \subset \Comb(P) \cong \Sym(m)$ of combinatorial automorphisms. Recover an H-representation of $P$ up to the action of $G$, i.e., recover a single representative of each facet orbit.
\end{problem}
The previous example with \numprint{136} extreme rays and \numprint{5221782341716704} facets, has a large linear automorphism group of order \numprint{660602880}. We will show that under these automorphisms there are \numprint{64001686} facet \emph{orbits}, which, while still large, becomes feasible to compute with our methods.

\section{Ryshkov's polyhedra and perfect forms}
\label{sec:ryshkov_s_polyhedra_and_perfect_forms}

There are two viewpoints on perfect forms that have been considered, which are essentially dual to each other.
The original one, introduced by Voronoi~\cite{VoronoiI}, considers \emph{perfect domains}, and is extensively treated by Martinet~\cite{martinet}.
The other point of view, which considers \emph{Ryshkov polyhedra},
%~\footnote{Technically, according to our definition the Ryshkov Polyhedra are not polyhedra as they are an \emph{infinite} intersection of halfspaces. They are however \emph{locally finite}: intersection with any polytope results in a polytope.}
was introduced in \cite{ryshkov1970polyhedron} and is detailed further in \cite{achill-enumerating}.
Both viewpoints are considered together in \cite{Opgenorth2001} in a generalized setting.
The Ryshkov polyhedron gives a more intuitive view with respect to the lattice packing problem, so we focus on this notion here.
The idea is to restrict the space of lattice bases (up to isometry), to those with a (squared) first minimum of at least some $\lambda > 0$.
It turns out that under the Frobenius inner product, this results in a convex body.
\begin{definition}[Ryshkov polyhedron~\cite{ryshkov1970polyhedron}]
For any $\lambda > 0$ we define the Ryshkov polyhedron $\P^d_{\lambda}$ with minimum $\lambda$ by
\begin{align*}
    \P^d_\lambda &= \{ Q \in \S_{>0}^d : \lambda(Q) \geq \lambda \} \\
    &= \bigcap\limits_{x \in \ZZ^d \setminus \{ 0 \}} \{ Q \in \S^d : \langle Q, xx^t \rangle \geq \lambda \} \subset \S_{>0}^d.
\end{align*}
\end{definition}
An example of a Ryshkov polyhedron in dimension $2$ is shown in~\cref{fig:ryshkov}. Note that Ryshkov polyhedra are technically not polyhedra according to our definition, as they are an \emph{infinite} intersection of halfspaces. However they are \emph{locally finite}~\cite[Theorem 3.1]{bookschurmann}: the intersection of $\P^d_\lambda$ with any other polytope $P$, is again a polytope, i.e., locally they behave the same as a polyhedron. In particular, we can therefore still naturally talk about vertices and facets of $\P^d_\lambda$, even if there are an infinite number of them.
With this parallel, note that the facets of $\P^d_\lambda$ correspond precisely to all primitive integer vectors $\pm x \in \ZZ^d$ up to sign.
\begin{figure}
\includegraphics[width=0.49\textwidth,keepaspectratio]{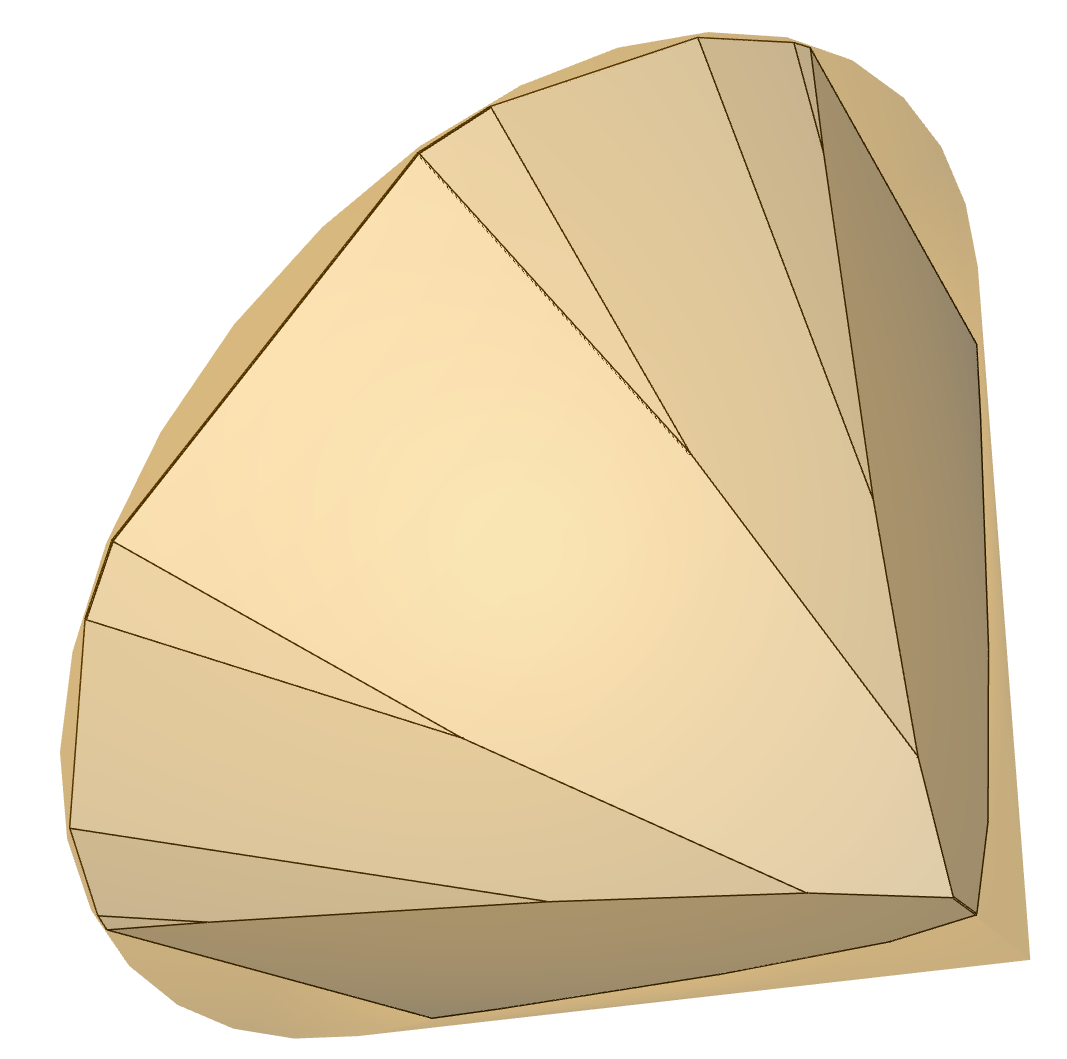}
\includegraphics[width=0.49\textwidth,keepaspectratio]{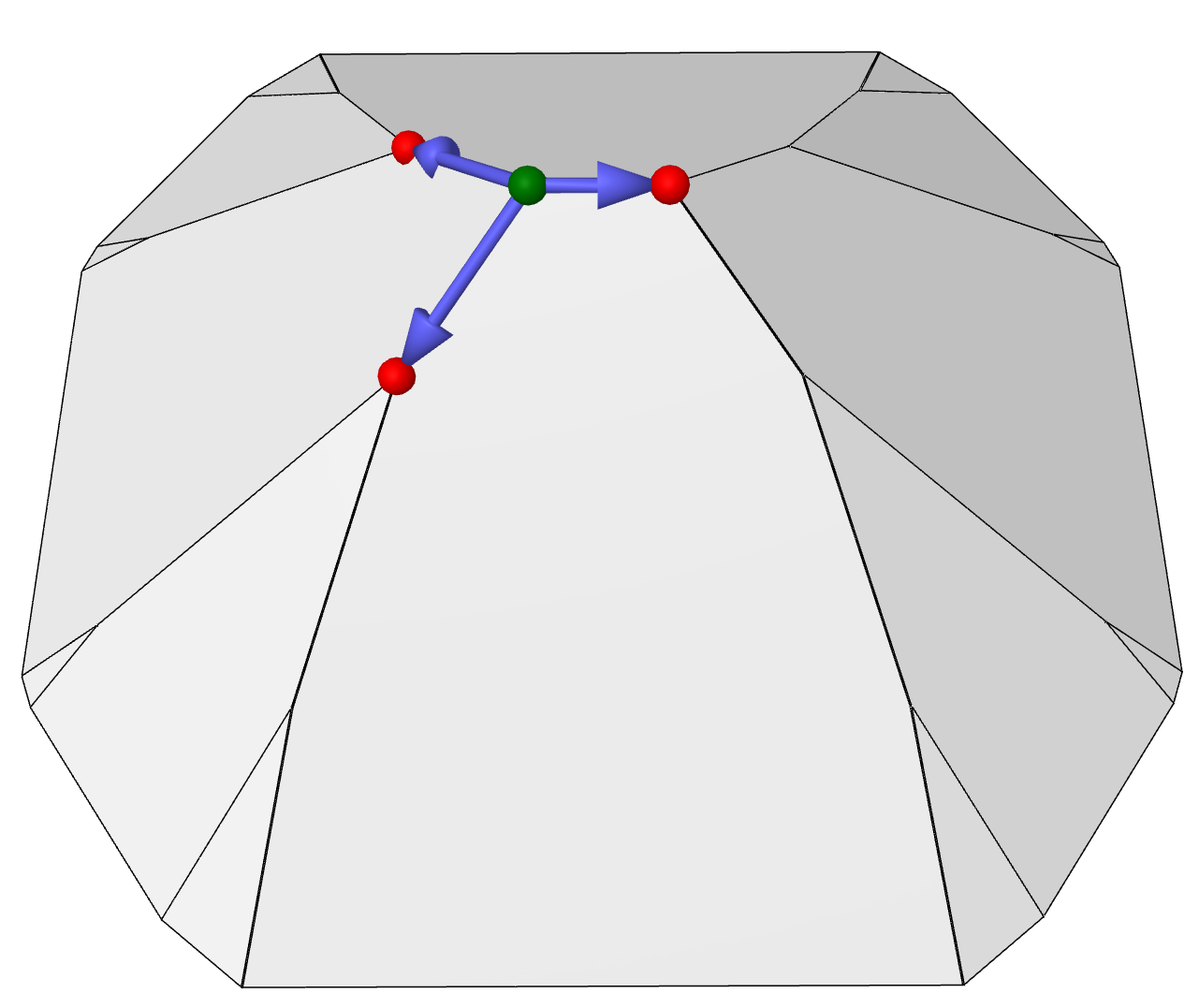}
\caption{On the left, part of Ryshkov polyhedron inside the cone of $2$-dimensional positive definite quadratic forms. Its vertices are perfect forms, such as the red and green vertices on the right. The group $\GL_d(\ZZ)$ acts as a symmetry group on the Ryshkov polyhedron. The red vertices are neighbouring perfect forms of the green vertex, indicating an exploration step of Voronoi's algorithm.
}
\label{fig:ryshkov}
\end{figure}

Now that the first minimum $\lambda_1(Q)$ is lower-bounded for all $Q \in \P^d_\lambda$, to maximize the Hermite invariant $\gamma(Q) = \lambda_1(Q)/\det(Q)^{1/d}$, it is sufficient to minimize the normalized determinant over $\P^d_\lambda$, which turns out to be a concave function.
\begin{lemma}[Convex optimization]
Let $\lambda > 0$, for the Ryshkov polyhedron $\P^d_\lambda \subset \S_{>0}^d$ we have
\begin{align*}
  \gamma_d = \frac{\lambda}{\min_{Q \in \P^d_\lambda} \det(Q)^{1/d}}.
\end{align*}
Furthermore, $Q \mapsto \det(Q)^{1/d}$ is concave over $\S_{>0}^d$ and thus $\gamma_d$ is attained at the vertices of $\P^d_\lambda$.
\end{lemma}
\begin{proof}
The function $Q \mapsto \det(Q)^{1/d}$ was already shown to be concave in $\S_{>0}^d$ by Minkowski~\cite{minkowski1905diskontinuitatsbereich} (see~\cite[Theorem 6]{vanWoerdenMasterThesis} for an elementary proof). It is even strictly concave except under scaling, in particular it is strictly concave on the boundary of $P_\lambda^d$.
Because $\P_\lambda^d$ is locally finite ~\cite[Theorem 3.1]{bookschurmann} and convex the minimum of the function over $\P_\lambda^d$ is thus attained at the vertices.
\end{proof}
To solve the lattice packing problem, i.e., to find the optimal value of $\gamma_d$ we thus can focus on the vertices of $\P^d_\lambda$.
These are precisely the \emph{perfect forms}.
\begin{definition}
A PQF $Q \in \S^d_{>0}$ is called \emph{perfect} or a \emph{perfect form} if one of the following three equivalent statements is true:
\begin{itemize}
\item $Q$ is a vertex of the Ryshkov polyhedron $\P^d_{\lambda_1(Q)}$, or
\item the set $\{ xx^t : x \in \Min(Q)\} \subset \S^d$ has full rank $n=\frac{1}{2} d (d+1)$.
\item $Q$ is uniquely determined by $\Min(Q)$ and $\lambda_1(Q)$.
\end{itemize}
\end{definition}
The second statement also has a geometric interpretation, namely the (inner) normal cone of $Q \in \P^d_{\lambda_1(Q)}$, also known as the \emph{Voronoi domain} $\Vor(Q)$ of $Q$, is spanned by $\{ xx^t : x \in \Min(Q) \} \subset \S^d$, and it has full dimension $n = \frac{1}{2}d(d+1)$ in $\S^d$ if and only if $Q$ is a vertex.
Perfect forms are thus indeed fully determined by their minimal vectors up to scaling, i.e. they are the unique solution $Q'=Q$ to the system $\langle xx^t, Q' \rangle = \lambda$ for all $x \in \Min(Q)$.

The Ryshkov polyhedron $\P^d_\lambda \subset \S_{>0}^d$ has an infinite number of facets and an infinite number of vertices. In particular, there are an infinite number of perfect forms. However, there is also an infinite symmetry group acting on $Q \in \P^d_\lambda$, namely the group of unimodular transformations $U \in \GL_d(\ZZ)$ which acts on the right by $Q \circ U \mapsto U^t Q U$.
Note that these transformations are linear on $\S^d$ and keep the first minimum invariant, and are thus automorphisms of the Ryshkov polyhedron $\P^d_\lambda$.
More precisely, a facet defined by a primitive vector $x \in \ZZ^d$ is mapped to a facet defined by $U^{-1}x$, and we have $\Min(U^tQU) = U^{-1}\Min(Q)$.
Under these symmetries, there are only a finite number of distinct vertices or non-similar perfect forms.
\begin{theorem}[Number of perfect forms~\cite{bacher2018number,van2020upper}]
Up to similarity there are only a finite number $p_d$ of perfect forms in each dimensions. In particular, we know that $\exp(\Omega(d)) < p_d < \exp(O(d^2 \log(d))$.
\end{theorem}
The exact number of perfect forms in dimension $1$ to $9$, the last being the result of this work, are displayed in~\cref{tab:nbforms}.
One can see that the number of them grows extremely fast, seemingly closer to the super exponential growth of the upper bound $\exp(O(d^2 \log(d))$, than the single exponential lower bound $\exp(\Omega(d))$.

Note that not every perfect form is necessarily a local optimum for the Hermite constant, the perfect forms that are, are called \emph{extreme} and they are characterized by a eutaxy property.
\begin{definition}[Eutaxy and extreme forms]
 A PQF $Q \in \S^d_{>0}$ is called \emph{extreme} if it attains a local maximum for the Hermite invariant $\gamma_d$.
 $Q$ is called eutactic if its inverse $Q^{-1}$ is contained in the relative interior of the Voronoi domain $\Vor(Q)$. Equivalently, there exist $\lambda_x > 0$ for $\pm x \in \Min(Q)$ such that we have a \emph{eutaxy relation} $Q^{-1} = \sum_{\pm x \in \Min(Q)} \lambda_x \cdot xx^t$.
\end{definition}
\begin{lemma}[{\cite{VoronoiI}}]\label{lem:extreme_eutactic}
  A PQF $Q \in \S^d_{>0}$ is extreme if and only if it is perfect and eutactic.
\end{lemma}
We can also consider weaker and stronger forms of eutaxy. For example, we call $Q$ \emph{semi-eutactic} if we also allow the scalars $\lambda_x$ to be $0$, and we call $Q$ \emph{strongly-eutactic} if there exists a eutaxy relation with all $\lambda_x$ identical.

\section{Voronoi's algorithm}
\label{sec:voronoi_s_algorithm}
The classical Voronoi algorithm \cite{VoronoiI} is a graph traversal algorithm
that enumerates all perfect forms. That algorithm is explained in the books \cite{bookschurmann,martinet}
and articles \cite{PerfectDim8,achill-enumerating}.

\subsection{The algorithm}
Voronoi's algorithm makes use of the fact that perfect forms, or equivalently, vertices of a Ryshkov polyhedron, are connected via edges to neighbouring vertices or what we call neighbouring perfect forms.
See \cref{fig:ryshkov} for an example in dimension $2$.
For every perfect form, starting for example with the perfect root lattice $\Alat_d$, the algorithm enumerates all neighbouring perfect forms and adds them to the list if they have not been found before (up to equivalence).
As every perfect form is connected via a path of edges to any other perfect form, and this is preserved by the equivalence, one will eventually enumerate all of them.
In particular, once all perfect forms in the list have been treated, as in their neighbours have all been enumerated and added, we know that the list must be complete and Voronoi's algorithm terminates with a complete enumeration of perfect forms in a fixed dimension.

\begin{center}
\begin{algorithm}[h!]
\caption{Voronoi's algorithm~\cite{VoronoiI}} \label{alg:voronoisalgo}
\begin{algorithmic}[1]
\Require Dimension $d \geq 1$.
\Ensure A complete list $L$ of non-similar perfect forms of dimension $d$.
\State $L \gets \{ Q_{\Alat_d} \}$ \Comment{perfect form corresponding to root lattice $\Alat_d$.}
\While{$\exists$ untreated $Q \in L$}
  \State Compute $\Min(Q)$.
  \State Enumerate the extreme rays $R_1, \ldots, R_k$ of the cone \begin{align*}
      \P(Q) = \{ Q' \in \S^d : \langle Q', xx^t \rangle \geq 0 \text{ for all } x \in \Min(Q) \}.
    \end{align*}
  \ForAll{extreme rays $R_i$ of $\P(Q)$}
    \State Determine neighbouring perfect form $Q_i = Q + \alpha_i R_i$. \Comment{$\alpha_i > 0$}
    \If{$Q_i$ is not equivalent to any form in $L$}
      \State $L \gets L \cup \{ Q_i \}$.
    \EndIf
  \EndFor
  \State Mark $Q \in L$ as treated.
\EndWhile
\State \Output $L$
\end{algorithmic}
\end{algorithm}
\end{center}

\subsubsection*{Dual description.}
\label{sub:dualdescription}
To compute the neighbouring perfect forms of a perfect form $Q \in \S_{>0}^d$ we first need to compute the edges adjacent to the vertex $Q$ in the Ryshkov polyhedron $\P_\lambda^d$ where $\lambda = \lambda_1(Q)$.
By definition of the Ryshkov polyhedron the facets $\P_\lambda^d$ adjacent to $Q$ correspond precisely to the minimal vectors $\Min(Q)/\{ \pm \}$ of $Q$ up to sign as $\pm x$ corresponds to the facet-defining inequality $x^t Q' x = \langle Q', xx^t \rangle \geq \lambda$ for $Q' \in \S^d$.
By translation we can in fact look at the pointed polyhedral \emph{tangent cone}
$$\P(Q) := \{ Q' \in \S^d : \langle Q', xx^t \rangle \geq 0 \text{ for all } x \in \Min(Q) \}$$
of $Q$, as in the algorithm.
Note that the tangent cone $\P(Q)$ is precisely dual to the inner normal cone $\Vor(Q)$ known as the Voronoi domain of $Q$.
The extreme rays of $\P(Q)$ correspond precisely to the edges, more precisely for each ray generator $R_i$ of $\P(Q)$ there exists a scalar $\alpha_i > 0$ such that $Q_i = Q + \alpha_i R_i$ is a neighbouring perfect form with the same minimum.

Computing the extreme rays of $\P(Q)$ from the known facets, or equivalently, computing the facets of $\Vor(Q)=\P(Q)^\circ$ from the known extreme rays, is precisely a dual description problem.
For certain cases this problem is relatively easy, for example when $|\Min(Q)| = d(d+1)$ there are precisely $\frac{1}{2}d(d+1)$ linearly independent facets equal to the dimension of the space $\S^d$. In this case $\P(Q)$ is a simplicial cone with precisely $\frac{1}{2}d(d+1)$ extreme rays, and computing these can be done relatively quickly.
Informally, we call these cases of \emph{low incidence}.
We can efficiently solve these low-incidence cases with specialized software like \textbf{LRS}~\cite{lrs}, \textbf{CDD}~\cite{cdd} and \textbf{PPL}~\cite{BagnaraHZ08SCP}.

When the number of minimal vectors up to sign, and thus the number of facets of $\P(Q)$ grows beyond the dimension of the space the number of extreme rays can however quickly grow, leading to a combinatorial explosion.
For these cases of \emph{high incidence} the number of extreme rays, and therefore the number of neighbouring perfect forms, is often so high that it becomes infeasible to enumerate them all.
However, often a lot of these neighbouring perfect forms are in fact equivalent.
For example, an automorphism $U \in \Aut(Q)$ permutes the minimal vectors $\Min(Q)$ and therefore permutes the facets of $\P(Q)$.
As this permutation respects the geometry of the minimal vectors it is in fact an automorphism of the polyhedral cone $\P(Q)$, i.e., we have an inclusion $\Aut(Q)/\{ \pm 1 \} \hookrightarrow \Comb(\P(Q))$.
More precisely, for any extreme ray generator $R$ of $\P(Q)$, there is another extreme ray $U^tRU$ of $\P(Q)$. And both rays point to equivalent forms as:
$$Q + \alpha R \sim U^t(Q+\alpha R)U = U^tQU + \alpha U^tRU = Q + \alpha U^tRU,$$
using that $U \in \Aut(Q)$.
We thus only have to compute the orbits of the extreme rays of $\P(Q)$ up to the automorphisms induced by $\Aut(Q)$.
This is therefore a dual description problem under symmetries.
Even with all these symmetries this remains the main computational cost of Voronoi's algorithm in dimension $9$.
We will further discuss how to treat this problem efficiently in~\cref{sec:the_dual_description}.

\subsubsection*{Finding neighbouring forms}
\label{sub:neighbouring_forms}
Given a perfect form $Q$ and an extreme ray $R$ of $\P(Q)$ the algorithm needs to compute a scalar $\alpha > 0$ such that $Q' = Q + \alpha R$ is a perfect form with the same minimum.
Note that for all $x \in \Min(Q)$ we have $x^tRx \geq 0$ by definition of $R$ being an extreme ray of $\P(Q)$.
Those minimal vectors that correspond to the facets adjacent to the extreme ray satisfy $x^tRx = 0$ and therefore will also end up in $\Min(Q')$.
We need to find the minimal $\alpha > 0$ such that $Q+\alpha R$ lies on a new facet, i.e., the minimal $\alpha > 0$ such that $\Min(Q')$ contains a new vector not contained in $\Min(Q)$.
One can approach this by a binary search, i.e., for such a new vector $y \in \Min(Q') \setminus \Min(Q)$ we must have $y^tRy < 0$ and thus for $\alpha' > \alpha$ we have that either $Q+\alpha' R$ is not positive definite, or that $\lambda_1(Q+\alpha' R) < \lambda_1(Q)$.
In the latter case we can also compute any $y \in \Min(Q + \alpha' R) < \lambda_1(Q)$ and decrease $\alpha'$ to the point that it satisfies $y^t(Q+\alpha'R)y = \lambda_1(Q)$.
For $\alpha' < \alpha$ we have $\Min(Q+\alpha' R) \subset \Min(Q)$ and thus we can also detect this.
This binary search process can be further optimized in practice using LLL and a fast implementation for the enumeration of short vectors.
For a more extensive treatment of this step, both from a practical and an asymptotic point of view we refer to ~\cite[Chapter 5]{vanWoerdenMasterThesis}.

\subsubsection*{Isomorphism checks.}
\label{sub:isomorphism_check}
When a neighbouring perfect form is found we need to check if it is already equivalent to one in our list, and otherwise we should add it.
A naive approach would check the new form for equivalence with every element in the list but this quickly becomes infeasible.
This can be improved further by grouping the list on invariants such as the determinant of the form, however this is still not sufficient for our scale, for example some determinants occur hundreds or even thousands of times in our list of non-similar perfect forms.
We therefore use a canonical function that returns a canonical representative and a canonical hash of the equivalence class of a form.
For more information see~\cref{sub:can_lattices}.

\subsection{Implementation details}
We implemented Voronoi's algorithm in highly-optimized \cpp{} code with the usage of several existing \cc{} and \cpp{} libraries.
We have used both \textbf{LRS}~\cite{lrs} and \textbf{PPL}~\cite{BagnaraHZ08SCP} for the dual description of the plain dual description cases without symmetry.
For the medium-incidence cases we preferred the low-memory usage of \textbf{LRS}, while \textbf{PPL} was typically a bit faster for the low-incidence cases.
We used \textbf{Traces}~\cite{mckay2014practical,nauty} to compute a canonical form (see~\cref{sub:can_lattices}), \textbf{FLINT}~\cite{flint} for large integer support and the computation of the determinant and Hermite Normal Form of a matrix, and MurmurHash3 as a fast hashing function~\cite{Appleby_MurmurHash3}.

We stored the list of perfect forms in a binary \textbf{netCDF}~\cite{NetCDF_Software} file. For each perfect form we store the upper diagonal integer coefficients of the scaled canonical form $Q$, and metadata such as $|\Min(Q)|$, $\lambda_1(Q)$, $|\Aut(Q)|$, the number of neighbours, its hash, its parents hash, and if it was marked as treated already. This adds up to a total of $147$ bytes per form or roughly $329$GB in total, which drops to roughly an average of $67$ bytes and $150$GB respectively after compression.

The dual description cases under symmetry were done with a separate implementation which we will explain more in \cref{sec:the_dual_description}. The output of this is a list of orbits of extreme rays for each high-incidence case and we processed these with specialized code on a case by case basis.

% \todo{Make public}

\subsubsection*{Parallelism via MPI}
\label{subsub:voronoi_mpi}
In \cite{IsoEdgeSixDim} an enumeration of another kind of geometrical object was done using a parallelized algorithm.
We follow a similar approach and explain here our implementation.
Our implementation is parallelized via the Message Passing Interface (MPI), where each process is behaving independent and communicates with other processes via messages.
This allows to relatively easily scale to hundreds of processes on multiple nodes.

We use the canonical form mentioned earlier to compute a hash of each perfect form.
The hash is used for two purposes:
\begin{enumerate}
\item By computing the residue modulo $n_{proc}$ we can partition the forms on a number of processors.
\item We can use a hash-map to look up a form in $O(1)$ amortized time.
\end{enumerate}
The partitioning induced by the hash function allows to evenly distribute the data and the processing to the available number of processors.
As a result every process only keeps a limited amount of data on disk and in memory.

Each process first checks if there are new incoming messages  and further processes these when needed.
Typically these are forms sent by other processes and the hash-map is used to efficiently recover potential candidates that are equivalent, after which the canonical forms are compared and the received form is inserted or not depending on the outcome.
If all incoming messages are processed an untreated form is taken from the local list and its neighbouring perfect forms are computed.
When a new form is found the canonical form and the hash are computed locally after which it is sent to the appropriate node determined by the hash.
The sending is done in batches per target process to limit the message overhead.
Furthermore, the few high-incidence forms are more likely to be found many times and therefore each process also keeps a local list to reject those early.

The processes stop after a set amount of time and after having finished all the ongoing tasks and having processed all messages.
At this point all the data are saved to disk such that the algorithm can easily continue later.
We typically ran the algorithm on $64$ to $256$ processes for a few days at a time.
All perfect forms with $|\Min(Q)| \leq 2 \cdot 58$ were treated in this way, which is about $99.9991\%$ of all forms.

\section{Canonical forms}
\label{sec:canonical_forms}
Canonical forms play a crucial role in this work.
They allow to efficiently distinguish up to billions of orbit equivalence classes in many different settings.
The general setting we consider is a group $G$ acting on a set $X$ by some (left or right) group action $g \circ x$ for $x \in X, g \in G$.
This subdivides the set $X$ into several orbit equivalence classes, where $x \sim y$ for $x,y \in X$ if and only if $y = g \circ x$ for some $g \in G$.

Then, given some subset $S \subset X$, the computational problem is to determine (a single representative of) all orbit equivalence classes $T := \{ \text{Orb}(G,x) : x \in S \}$ in $S$. I.e., we want to remove all duplicates from $S$ up to equivalence.

Let $|S|, |T|$ be the size of $S$ and $T$ respectively. A naive approach to determine $T$ would require roughly $\frac{1}{2}|S|^2$ pairwise equivalence checks to remove any duplicates.
By building $T$ incrementally this can be improved to at most $|S|\times |T|$ pairwise equivalence checks.
In this work, in several different settings, we have to deal with sets of size $|S| \gg |T| > 10^9$, making neither approach feasible.

The idea instead is to use a canonical function $\Theta : X \to X$ that satisfies $\Theta(x) \sim x$ and
$$x \sim y \Leftrightarrow \Theta(x) = \Theta(y) \quad \text{for all } x,y \in X.$$
In other words, a canonical function returns a canonical representative for each orbit class, turning equivalence checks into equality checks.
One can thus simply compute $T = \{ \Theta(x) : x \in S \}$ using $|S|$ canonical function computations, and with an additional $O(|S|)$ amortized cost for removing duplicates using a hash-map.
Overall, if we can compute such a canonical function, we remove the time complexity in terms of $|S|$ from quadratic down to linear.
This makes the method practically feasible up to billions of elements.

We use such canonical functions in the context of lattice isomorphisms, permutation groups acting on sets, and polyhedral equivalences.

\subsection{Graphs}
One of the most researched areas for canonical functions is that of graph isomorphisms. Here we consider the set $X$ of all weighted complete graphs $(\{1, \ldots, n\},\omega)$ with $n$ vertices labeled $1, \ldots, n$, along with a weight function $\omega : \{1, \ldots, n\}^2 \to \RR$ indicating the edge weight $\omega(i,j)$ between vertices $i$ and $j$.
A permutation group $G = \Sym(n)$ acts on such a graph by permuting the labels (and changing the weight function appropriately).
Two graphs that lie in the same orbit under this action are called \emph{isomorphic}, i.e., this is the case if one can permute the vertices such that the two graphs attain exactly the same weights.
The \emph{graph isomorphism problem} asks to find such a permutation between two isomorphic graphs.

On the theoretical side, in a groundbreaking work, Babai~\cite{babai2016graph} showed that one can always solve the graph isomorphism problem in quasi-polynomial time $\exp(\log(n)^c)$ in the number of vertices.
Furthermore, this result was extended to a canonical function for graph isomorphism~\cite{babai2019canonical}.
More importantly for us, on the practical side, there are many highly optimized implementations of such canonical functions, such as the libraries \textbf{Nauty}, \textbf{Traces}~\cite{mckay2014practical,nauty} and \textbf{Bliss}~\cite{bliss}, that can work on graphs with up to thousands of vertices.
Typically these implementations are limited to a binary weight function with support $\{ 0, 1 \}$, but there exist efficient reductions to the binary case using $n \cdot \sqrt{\log_2(|\text{supp}(\omega)|)}$ vertices. See the \textbf{Nauty} manual~\cite{mckay2025nauty} for more details.

We will not be using graphs and graph isomorphism directly in this work, but we will use it as a tool to compute other canonical functions.

\subsection{Lattices and PQF's}
\label{sub:can_lattices}
Voronoi's algorithm finds many perfect forms, and an important requirement for the termination of the algorithm is to only keep those that are new, i.e., those that are non-equivalent to any other perfect form already in our database.
As explained in \cref{sub:quadratic_forms_and_the_hermite_constant}, the set of all PQF's of dimension $d$ can be described by $\S_{>0}^d$, and we have a (right) action by the group $\GL_d(\ZZ)$ of unimodular matrices given by $Q \circ U = U^t Q U$ for $Q \in \S_{>0}^d$ and $U \in \GL_d(\ZZ)$, defining an equivalence relationship.

In~\cite{PerfectDim8,vanWoerdenMasterThesis} the number of perfect forms enumerated was small enough to use a combination of some equivalence invariants, like the determinant, and pair-wise equivalence checks using the backtracking-search approach by Plesken and Souvignier~\cite{PleskenSouvignier}.
However, this approach is far from sufficient for the billions of perfect forms we have to handle, as many of them have overlapping invariants.

Motivated by this problem we developed a canonical function for quadratic form equivalence, which was first discussed in this context in~\cite{vanWoerdenMasterThesis} and later published in~\cite{CanonicalFormPositiveForm}.
The high-level idea is to reduce the canonical form computation of a PQF to that of a graph.
To see how this works let's consider an automorphism $U \in \Aut(Q)$ and the set of minimal vectors $\Min(Q)$ of a PQF $Q$. By definition we have $\Min(Q) = \Min(U^tQU)$, but also $\Min(U^tQU) = U^{-1} \cdot \Min(Q)$. Automorphisms thus induce a permutation on the set of minimal vectors. Moreover, this permutation is an isometry w.r.t. the inner product induced by $Q$, i.e.,
$$(U^{-1}x)^t Q (U^{-1}y) = (U^{-1}x)^t (U^tQU) (U^{-1}y) = x^t Q y \quad \text{ for all }x,y \in \Min(Q).$$
We can encode this in a weighted graph $\cG_Q$ as follows: pick some ordering $v_1, \ldots, v_m$ of the minimal vectors $\Min(Q)$, and define the weight function by $\omega(i, j) = v_i^t Q v_j$.
Clearly, any automorphism of $Q$ then induces an automorphism on the weighted graph $\cG_Q$. Moreover, if $Q' \sim Q$, then $\cG_Q \sim \cG_{Q'}$.

One can show that the reverse implication is also true if $\spa_{\ZZ}(\Min(Q)) = \ZZ^d$.
In that case the problem of PQF equivalence thus becomes one of graph equivalence. Furthermore, any canonical function for graph equivalence can in turn give a canonical form.
From our enumeration it follows that the property $\spa_{\ZZ}(\Min(Q)) = \ZZ^d$ is attained for all but one $9$-dimensional perfect lattice and thus using that set is almost always sufficient. We check for this condition and otherwise one has to consider a set of slightly larger vectors.
For more details we refer to~\cite{CanonicalFormPositiveForm}. A similar approach has later been used in~\cite{IsoEdgeSixDim} to build a canonical function for C-type domains.

To further improve the canonical function from~\cite{CanonicalFormPositiveForm} we also consider the absolute trick introduced in~\cite{van2023lattice,IsoEdgeSixDim}. The idea is that any minimal vector $x \in \Min(Q)$ always comes in a pair $\pm x$.
We therefore build an \emph{absolute graph} with only half the original vertices, typically corresponding to the minimal vectors up to sign $\pm x \in \Min(Q) / \{ \pm \}$, and by replacing the weight function by $w(i,j) = |v_i^t Q v_j|$ for $\pm v_i, \pm v_j \in \Min(Q) / \{ \pm \}$.
The absolute graph is half the size of the original graph and thus we can expect the canonical graph algorithms to be more efficient.

The main problem is that an isomorphism between absolute graphs does not always induce one between the corresponding PQF's.
One can however check this given the automorphisms of the absolute graph that one gets for free from the canonical form computation. If each of them lifts correctly to an automorphism of the PQF we still obtain a canonical form. If not one can still run the computation with the full graph, although this is rarely needed for the perfect forms we consider. For more details see~\cite[Chapter 9.7.2]{van2023lattice}.
When using the \textbf{Traces} library for computing the canonical function on a graph one obtains a speed-up of roughly a factor $3$ from the absolute trick as shown in \cref{tab:latcanonic}.
In the same table we also see that the computation of the absolute canonical form is only between a factor $1.1$ and $1.6$ times slower than the computation of the automorphism group using the program \textbf{AUTO} by Plesken and Souvignier~\cite{PleskenSouvignier}, which is comparable to doing a single isomorphism check.
\begin{table}
\caption{Average timing of regular and absolute canonical quadratic form computation over all $9$ dimensional perfect forms versus the computation of the automorphism group using the \textbf{AUTO} program~\cite{PleskenSouvignier}. The second row displays the weighted average by (orbit) number of neighbours which represents the ratio in which these forms are encountered during Voronoi's algorithm. The benchmarks were performed on a Xeon Gold SKL-6130 CPU at 2.1 Ghz.}
\label{tab:latcanonic}
\begin{tabular}{llll} \toprule
Forms & Canonical & Absolute & \textbf{AUTO} \\ \midrule
\numprint{2237251040} $9$-dim perfect forms & $913\unit{\us}$ & $304\unit{\us}$ & $194\unit{\us}$
 \\
 adjusted for \# neighbour orbits & $1184\unit{\us}$ & $427\unit{\us}$ & $372\unit{\us}$ \\
 \bottomrule
\end{tabular}
\end{table}

Lastly, we used the results by Watson~\cite{G1971} which claim that the $9$-dimensional perfect forms with $2 \cdot 129$ and $2 \cdot 136$ minimal vectors respectively are unique, to quickly differentiate those.

\subsection{Polytopes}
\label{sub:can_polytopes}
If two pointed polyhedral cones are combinatorially equivalent and this equivalence is known, one only has to solve the dual description problem for one of the cones.
From the equivalence one then directly obtains the dual description of the other cone.
In \cref{sub:saving_bank} we use this principle to limit the number of costly dual description computations.
Again we would like a canonical function to quickly determine if we have already seen an equivalent cone before.

As combinatorial automorphisms and equivalence are typically hard to compute we will only consider restricted linear equivalence.
For this we have a group of linear transformations $G = \GL_n(\RR)$ acting on sets $\{ v_1, \ldots, v_m \} \subset \RR^n$ of full-rank ray generators.
Just as for the equivalence of PQF's we consider a reduction to the case of weighted graphs following~\cite{GroupPolytopeLMS}.
In this work it is shown that if one computes the symmetric matrix $S = \sum_{i=1}^m v_i v_i^t$, and constructs an $m$-vertex graph with weight function $\omega(i,j) = v_i^t S^{-1} v_j$, that then restricted linear automorphisms and isomorphisms correspond to the respective automorphisms and isomorphisms on these graphs.

Again, by computing a canonical representative of the graph we can also construct a canonical representative for restricted linear equivalence.
The idea is that the canonical graph representative gives a canonical ordering of the extreme ray generators which is unique up to $\GL_n(\RR)$-linear transformations. By Gaussian elimination one then obtains a canonical representative.

This canonical function is also used to discover which of the high-incidence cones $\P(Q_1), \P(Q_2)$ for distinct perfect forms $Q_1, Q_2$ are equivalent.
For example, the \numprint{19155} perfect forms $Q$ with $|\Min(Q)| \geq 2 \cdot 59$, have only \numprint{7441} distinct Voronoi domains up to restricted linear equivalence.

\subsection{Permutation groups and face equivalence}
\label{sub:can_face}
Recall from \cref{sub:polytopes} that faces of a pointed polyhedral cone can be represented by subsets of $\{ 1, \ldots, m \}$ representing the $m$ extreme rays or $m$ facets, and that any automorphism group on the cone can be represented as a subgroup of $\Sym(m)$ acting on a face $S \subset \{ 1, \ldots, m \}$ by $\sigma \circ S = \{ \sigma(i) : i \in S \}$.
For computations related to face equivalence we can thus focus on permutation groups acting on subsets. These computations will be important in our algorithm to solve the dual description problem under symmetry further treated in~\cref{sec:the_dual_description}.

The computer algebra system \textbf{GAP}~\cite{GAP} implements many permutation groups algorithm which are very efficient and useful.
% This is because a facet of a polytope can be encoded by the subset of the vertices contained in it.
In particular the partition backtracking search~\cite{Leon1,Leon2} is very useful for testing equivalence of faces,
and for computing their stabilizer under the symmetry group of a polytope.
This was used in \cite{ComplexityVoronoiDSV,ContactLeech,CUT8_facet,montreal} for computing the dual description of a variety of polytopes, and in particular for completing Voronoi's algorithm in dimension $8$~\cite{PerfectDim8}.
For this work, working in \textbf{GAP} was not efficient enough and difficult to integrate with the rest of our {\tt C++} implementations.
Therefore, the \textbf{GAP} permutation group algorithms that were relevant to the dual description problem were translated to {\tt C++} in a new library \textbf{Permutalib}~\cite{permutalib}. This implementation was already successfully used in~\cite{BirkhoffDualDesc}.

In this work we have to determine the equivalence of up to billions of faces and thus again require a canonical function.
Because each orbit is finite one could in principle use the minimal orbit element under some predetermined ordering as a canonical representative.
An example of this is to use the standard lexicographic ordering.
However, the search for the minimal element under such an ordering can be very costly, depending on the properties of the group and the orbit.
One can improve this by (incrementally) determining a suitable ordering depending on the orbit in question, minimizing the backtracking search in the algorithm.
Such a canonical function was developed in~\cite{jefferson2019minimal} and implemented as a \textbf{GAP} package.
A generalization of the canonical function code from this \textbf{GAP} package was also translated to \textbf{Permutalib}~\cite{permutalib}.

For a group $G$ acting on subsets of $\{ 1, \ldots, m \}$, a subset $S \subset \{ 1, \ldots, m \}$ and a subgroup $H\subset G$ stabilizing $S$, it computes a canonical function $\text{Can}(G, H, S)$ for $S$.
For $\sigma\in G$ this canonical function has the following property:
$$\text{Can}(G, \sigma H \sigma^{-1}, \sigma \circ S) = \text{Can}(G, H, S).$$
The group $H$ is used to limit the backtracking search tree in the algorithm, but has to be fixed a priori in a canonical way.
In our context, we have two possibilities for the subgroup $H$: the stabilizer $\Stab(G,S)$ of $S$ under $G$ (the choice in~\cite{jefferson2019minimal}) \emph{or} the trivial subgroup. The corresponding canonical functions can be different.
If the stabilizer is large and one uses the trivial subgroup we could be exposed to memory and runtime problems. If one uses $\Stab(G, S)$ then this problem can be addressed but the computation of the stabilizer is expensive.

What we do is use the canonical function with a trivial stabilizer first. In most cases, the stabilizer is trivial so this approach is sufficient. The explosion in the number of cases in a layer of the search tree is handled in the following way.
If the number of cases exceed a specified threshold of $500$, then we compute the stabilizer and compute the canonical function for this stabilizer. Since the number of cases in a layer of the search tree is invariant under conjugation, this is indeed a canonical function.

\section{The dual description problem under symmetries}
\label{sec:the_dual_description}
Just as in dimension $8$, the dominating cost of Voronoi's algorithm in dimension $9$ is the dual description problem for the few polyhedral cones $\P(Q)$ of high incidence, or equivalently the dual description for their Voronoi domains $\Vor(Q) = \P(Q)^\circ$.
In this section we explain our computational methods and implementation efforts for solving the dual description problem under symmetries, which could be of independent interest.
For ease of intuition, we will consider the computation of an H-representation up to symmetries given a V-representation $V \in \QQ^{n \times m}$ of the cone $P = \{ V \lambda \in \RR^n : \lambda \in \RR_{\geq 0}^m \} $. Recall that by duality this is equivalent to the dual description problem in the other direction.
The symmetry group $G$ is assumed to be represented as a subgroup of the permutation group $\Sym(m)$ acting on the $m$ extreme rays.

\subsection{Recursive Adjacency Decomposition Method}
To compute the dual description under symmetries we use the \emph{Recursive Adjacency Decomposition Method}.
The Adjacency Decomposition Method (ADM) has been reinvented many times for example by Jaquet~\cite{jaquet1993enumeration} for the treatment of $\Vor(Q_{\Elat_7})$ in Voronoi's algorithm and by Christof and Reinelt~\cite{CR_small_polytopes} for cut polytopes, traveling salesman polytopes and the linear ordering polytopes. The algorithm and its name were formally introduced in~\cite{CR_decomposition_parallelization}.

On a high-level the idea behind the algorithm is similar to the graph traversal algorithm behind Voronoi's algorithm, which enumerates all perfect forms by repeated finding neighbouring perfect forms.
For a pointed polyhedral cone $P$ we call two faces $F_1, F_2$ of dimension $k > 0$ \emph{adjacent} if the face $R = F_1 \cap F_2$ has dimension $k-1$, i.e., if $\dim(F_1 \cap F_2) = \dim(F_1)-1 = \dim(F_2)-1$.
When considering the facets of $P$ as graph vertices\footnote{For polytopes and in a dual setting, the notions of vertices and edges for this adjacency graph coincide with their polytope counterparts, just as for Voronoi's algorithm in the setting of Ryshkov polyhedron.}, and their adjacency as edges, the facets form a connected graph which, just as for Voronoi's algorithm, can be explored up to equivalence.
The Adjacency Decomposition Method thus proceeds as follows: starting with a single facet, it will repeatedly enumerate all adjacent facets and keep those that are new up to equivalence, until all facets have been treated. See \cref{alg:ADM}.

An important step in the algorithm is the determination of all adjacent facets $F_1, \ldots, F_N$ of a facet $F$.
These are in a one-to-one correspondence with the ridges $R_1, \ldots, R_N$ of $P$ for which $R_i = F_i \cap F$. Given any ridge $R_i$ included in $F$ one can uniquely determine $F_i$ in a process called \emph{flipping} which we will detail in \cref{ssub:flipping_modulo_q}.
We can thus focus on obtaining the ridges $R_i$ of $P$ which are included in the facet $F$. When viewing the facet $F$ as a subcone (restricted to its linear span), the ridges $R_1, \ldots, R_N$ are precisely the \emph{facets} of $F$.
The determination of the ridges $R_1, \ldots, R_N$ is thus again a dual description problem, but of the subcone $F$ which is in a dimension lower, and given by only those extreme rays of $P$ that are contained in $F$. This typically leads to much easier instances that can be solved by standard dual description algorithms. See the face lattice in \cref{fig:RADM} which indicates the adjacent facets and the subcone defined by a facet.
In several cases these subcones might still be hard to process, therefore~\cite{PerfectDim8} extended the method to a recursive one.
On a closer inspection, it is sufficient to enumerate the facet orbits of $F$ up to the symmetries induced by $\Stab(G, F) \subset G$, i.e., any facets $R_i$ of $F$ in the same orbit under $\Stab(G, F)$ will lead to an adjacent facet of $P$ in the same orbit under $G$.
This is thus again a dual description problem under symmetry which we can solve using ADM.
This recursive approach allowed the treatment of the Voronoi domain $\Vor(Q_{\Elat_8})$ of the $\Elat_8$ root lattice in about a month of computation, and thereby led to the successful classification of all perfect forms in dimension $8$.

\begin{figure}
  \includegraphics[width=0.7\textwidth]{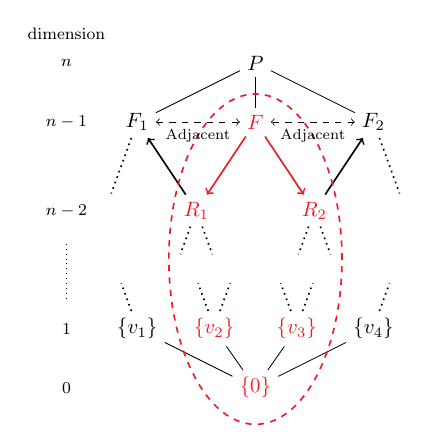}
  \caption{This figure shows an abstract view of (part of) the face lattice of a pointed polyhedral cone $P$ of dimension $n$. Face inclusions are indicated with solid lines. The facet $F$ is adjacent to $F_1$ and $F_2$ as they share ridges $R_1$ and $R_2$ respectively. The facet $F$ defines a subcone indicated in red with extreme rays $v_2, v_3 \in F$. The ridges $R_1$ and $R_2$ of $P$ are facets of $F$.}
  \label{fig:RADM}
\end{figure}

\begin{center}
\begin{algorithm}[h!]
\caption{Recursive Adjacency Decomposition Method~\cite{CR_decomposition_parallelization,PerfectDim8}} \label{alg:ADM}
\begin{algorithmic}[1]
\Require A V-representation $V \in \QQ^{n \times m}$ of a pointed polyhedral cone $P$, and a group $G \subset \Comb(P) \subset \Sym(m)$.
\Ensure A complete list $\cF$ of $G$-inequivalent facet of $P$.
\State $\cF \gets \{ F_{\text{start}} \}$ \Comment{with $F_{\text{start}}$ any facet of $P$.}
\While{$\exists$ untreated $F \in L$}
  \State Compute the facets $R_1, \ldots, R_N$ of $F \subset \spa_\RR(F)$. \Comment{direct or recurse}
  \ForAll{$R_i$ (up to $\Stab(G, F)$)} \Comment{ridges of $P$ included in $F$}
    \State Compute unique facet $F_i$ of $P$ such that $R_i = F \cap F_i$. \Comment{flipping}
    \If{$F_i$ is not equivalent to any facet in $\cF$}
      \State $\cF \gets \cF \cup \{ F_i \}$.
    \EndIf
  \EndFor
  \State Mark $F \in \cF$ as treated.
\EndWhile
\State \Output $\cF$

\end{algorithmic}
\end{algorithm}
\end{center}

\subsection{Algorithmic improvements}
We will now discuss some further algorithmic improvements.

\subsubsection{Canonical face representative}
Just as in Voronoi's algorithm one needs to differentiate efficiently between up to billions of facet orbits.
The idea of using a canonical function for this was already introduced in~\cite{CR_decomposition_parallelization}.
In~\cite{vanWoerdenMasterThesis} it was suggested to use the improved canonical function function from~\cite{jefferson2019minimal}.
We implemented this canonical function, with further improvements as detailed in \cref{sub:can_face}, and used it in our Recursive ADM algorithm.

\subsubsection{Early termination criterion}
The computation with the dual description requires computing the adjacent facets of each facet orbit.
This computation is done from the orbits with lowest incidence till the one of highest incidence.
It is often the case that the orbits of highest incidence are the most computationally intensive, but do not give any new orbits.
If we would not treat these cases we are however not sure if we have indeed found all orbits, as some orbits might only be connected to these.
Therefore, it is very useful to have criteria that guarantees we have found all orbits, and thus allow for an early termination of the enumeration.

\begin{theorem}[Connectivity theorem~\cite{BalinskiTh}]\label{thm:balinski}
Let $P$ be a pointed polyhedral cone of dimension $n$ and $\cG$ its facet adjacency graph. Then the removal of any $n-2$ facets from $\cG$ leaves the graph connected.
% If a polytope is $d$-dimensional, then the removal of any $d-1$ vertices from the skeleton leaves it connected.
\end{theorem}

Therefore, if the set of facets in untreated orbits has size at most $n-2$ then the untreated orbits do not need to be treated since they would not give anything new.
This strategy has first been used in \cite{rigid_dim_six}.
If one inspects the proof of Balinski's theorem then one sees that only the dimension is used.
I.e., given a set of facets and their facet-defining normal vectors $y_1, \ldots, y_N$, if these vectors span a subspace of rank at most $n-2$,
then the removal of those facets leaves the graph connected. Therefore if the facets in the untreated orbits have this property we can terminate the enumeration.

Looking even closer into the proof of~\cref{thm:balinski} given by Gr{\"u}nbaum~\cite{grunbaum1967convex}, there is a slightly stronger statement in case the set of facets contain a common extreme ray. In the context of the dual description problem this was mentioned in \cite[Theorem 1]{CUT8_facet}: if a set of facets all contain the same extreme ray, then removing them leaves the adjacency graph connected.

Note that the difference with the general Balinski statement is that these facets contain a \emph{common face} of dimension $1$, while for Balinski's statement the unique hyperplanes going through each facet must have a \emph{common subspace} of dimension $2$, that does not have to define a face.

\subsubsection{Extra symmetries and double coset decomposition}
\label{ssub:extra_symmetries_and_double_coset_decomposition}
While we typically need to compute the dual description under a certain symmetry group, for example those induced by $\Aut(Q)$, there might be more symmetries in the resulting cones.
When computing the dual description of such a complex cone, it is essential to use as many symmetries as possible.
An example of this is the Voronoi domain $\Vor(Q_{\Lambdalat_9})$, which has only \numprint{5160960} automorphisms induced by $\Aut(Q_{\Lambdalat_9}$, but \numprint{660602880} restricted linear automorphisms.
The same can be true in the Recursive ADM algorithm, when treating a facet $F$ recursively one needs to enumerate its facet orbits up to $\Stab(G,F)$.
However, the subcone induced by $F$ might have additional symmetries.
We try to find such extra symmetries whenever we think the dual description could be difficult, based on some heuristic rules.

This leaves one problem: we need to compute the dual description for a specified symmetry group. That is, we need to switch from the list of orbits up
to a symmetry group $G_1$ to the list of orbits for a smaller symmetry group $G_2$.
The method is the following. For an orbit $G_1x$ we compute the stabilizer $\Stab(G_1, x)$
of $x$ under $G_1$. Then we find the double cosets decomposition:
\begin{equation*}
G_1 = \cup_{i=1}^{N} G_2 g_i \Stab(G_1,x) \mbox{ for some } g_i \in G_1.
\end{equation*}
The orbit $G_1x$ then splits into the orbits $G_2 g_1 x, \ldots, G_2 g_N x$.

We used a number of algorithms for computing this orbit splitting. One is making the
generators of $G_1$ act on $x$ and test for $G_2$ equivalence by using the permutation
backtrack algorithm. This usually works, but sometimes the total number of orbits
is incomplete. This is detected by keeping track of the total number of elements.
If some orbits are missing, then we generate all the elements of $G$ until we have found
all the orbits.

Another technique is to compute the cosets $g_i$ of $G_2$ over $G_1$ so that $G_1 = \cup_{i=1}^N G_2 h_i$.
We then form the orbits $G_2 h_i x$ and eliminate by the partition backtrack, or by our canonical function as explained in~\cref{sub:can_face}, the ones that are isomorphic. This works well if the index of $G_2$ in $G_1$ is low.
These approaches were sufficient for our purpose.

After the enumeration of this work was completed, we finally implemented ascending chains of subgroups and double coset enumeration in \textbf{Permutalib}~\cite{permutalib} inspired by the implementation of \textbf{GAP}~\cite{GAP}. This turns out to be in most cases the most efficient technique.

\subsubsection{Saving Bank}\label{sub:saving_bank}
Suppose that we compute the dual description of facets $F_1$ and $F_2$ by using the adjacency decomposition method and that $F_1$ and $F_2$ share a common ridge $R$.
Then the dual description of $R$ could be computed two times which is a waste of
resources.
In case the algorithm recurses the dual description of lower-dimensional faces this could even increase more.
Furthermore, some of the ridges $R$ could simply be equivalent as a polyhedral cone and thus have equivalent dual descriptions.
Therefore, we have a memoization process which stores the computed dual descriptions in a \emph{saving bank} if they were expensive to compute, following the approach from~\cite{PerfectDim8}.
To quickly determine if an equivalent dual description has been stored already, up to restricted linear equivalence, we use the canonical form as detailed in \cref{sub:can_polytopes}.
If the dual description under symmetry has been stored with a different symmetry group we use the double coset decomposition techniques treated in~\cref{ssub:extra_symmetries_and_double_coset_decomposition}.

A good example of the usefulness of the saving bank can be seen in \cref{fig:face_lattice_high_incidence}, where the high-incidence ridge with $83$ extreme rays is up to equivalence shared by $6$ different facet orbits of $\Vor(Q_{\Lambdalat_9})$, and even by other cones like $\Vor(Q_{99})$ and $\Vor(Q_{129})$. Also some of the facets are shared between these cones up to equivalence.
For the most complex cases we treated, thousands of dual descriptions were stored and reused. This data also allowed for a closer inspection of the face lattice as for example shown in \cref{fig:face_lattice_high_incidence} and used in the proof of \cref{thm:poss_kissing_numbers}.

\subsection{Implementation details and improvements}
The algorithmic improvements alone were not sufficient, so we will now discuss some of the implementation details and improvements.
Ignoring the improvements to parallelism, our implementation effort along with the algorithmic improvements reduced the dual description computation for $\Vor(\Elat_8)$ from several months, as reported in~\cite{PerfectDim8}, down to only $9$ hours on a single core. Additionally, our highly parallel implementation made the much larger cases brought up by Voronoi's algorithm in dimension $9$ feasible, on which we will report more in \cref{sec:results}.

\subsubsection{\cpp{} implementation}
For the classification of the perfect forms of dimension $8$ the main bottleneck was the dual description of $\Vor(\Elat_8)$ which was performed using an implementation of Recursive ADM in the Polyhedral package~\cite{Polyhedral} in the \textbf{GAP} system for computational discrete algebra~\cite{GAP}.
While \textbf{GAP} has very efficient algorithms for permutation groups it is not sufficiently fast and capable for the large scale of our computation.
Therefore, as part of this work the Recursive ADM algorithm, along with all necessary permutation group features of \textbf{GAP}, have been ported to a \cpp{} implementation. This implementation has been developed publicly at~\cite{PolyhedralCpp}.
It depends on the \textbf{Eigen}, \textbf{GMP} and \textbf{Boost}~\cite{eigenweb,GMP3,boostweb} libraries for handling matrices, exact arithmetic, and many miscellaneous functions such as serialization respectively. Furthermore, it depends on \textbf{Nauty}~\cite{nauty} for the computation of canonical graphs, and uses \textbf{PPL}~\cite{BagnaraHZ08SCP} or \textbf{LRS}~\cite{lrs} for basic dual description computations.

\subsubsection{Incidence to face-defining inequality conversion modulo $q$.}
\label{ssub:flipping_modulo_q}
Recall that for a facet $F$ and a ridge $R_i$ contained in $F$, we can obtain an adjacent facet $F_i$ to $F$ satisfying $R_i = F \cap F_i$ in a process called flipping.
The idea is similar to that of finding neighbouring forms in Voronoi's algorithm, but simplified because there are only a finite number of extreme rays. First consider face-defining inequalities $f(x) := \langle y, x \rangle \geq 0$ and $f_i(x) := \langle y_i, x \rangle \geq 0$ represented by $y, y_i \in \RR^n$ for $F$ and $R_i$ respectively.
To determine the facet-defining inequality of $F_i$ we then simply have to find the minimal scalar $\beta < 0$ such that $f_i(v_i) + \beta \cdot f(v_i) \geq 0$ for all extreme rays $v_1, \ldots, v_m$ of $P$. Note that for this one only needs to consider those extreme rays that are not already in $F$.

When computing the ridges $R_i$ with a direct dual description method that does not consider any symmetries, such as with \textbf{PPL}~\cite{BagnaraHZ08SCP} or \textbf{LRS}~\cite{lrs}, one often directly obtains both the extreme rays included in $R_i$, and additionally the face-defining inequality $f_i$ of $R_i$. The flipping process, which uses both these data, can then immediately be applied.

However, in the more complex settings where we only have a large list of orbit representatives for the ridges $R_i$ up to $\Stab(G,F)$, we typically only keep the former representation due to storage constraints.
Storing only a bitstring of length $|F|$ is much more efficient than additionally storing a face-defining inequality in $\QQ^n$.
We thus need a method to quickly obtain a face defining inequality from the list of extreme rays included in the face.
In principle, this can be done by computing the kernel of the matrix given by those extreme rays.
This is a computation that has to be executed many times and computing the kernel of a rational matrix is relatively slow, especially because we need to use arbitrary precision rationals for correctness.

We noticed however for our cases that the face-defining inequalities can typically be given by integer vectors with relatively small entries.
Therefore, we computed a kernel vector $a \in \FF_p^n$ modulo some large prime $p \approx 2^{31}$, which is much faster as operations over this field can be implemented efficiently using $64$-bit integer arithmetic.
Then, using rational lifting, via continuous fraction methods~\cite{wang1982p} or Lagrange reduction, we lift each vector coefficient $a_i$ in $\FF_p$ to an appropriate rational $\frac{b_i}{c_i} \in \QQ$ such that $a_i = b_i c_i^{-1} \bmod{p}$ and such that $b_i$ and $c_i$ are small ($|b_i|, |c_i| < \sqrt{p}$). We then remove the common denominator to obtain an integer vector.
Typically this vector lies in the kernel which we check. In case the check fails we fall back to the original computation over the rationals, however we have not observed that this was required.
The resulting conversion technique leads to a $10\times$ or larger speed-up compared to the original approach.

\subsubsection{Parallel adjacency method via MPI}
Just as for our implementation of Voronoi's algorithm we used MPI~\cite{MPIforum} to scale the algorithm up to hundreds of processors.
On a high level our approach is also similar to our approach there so we will be brief about it.
We mostly parallelized the Recursive ADM algorithm on the top level, i.e., the facet orbits are distributed over the processors and treated separately.
The canonical function to represent the facets detailed in ~\cref{sub:can_face} is, along with a hash function, used to send each facet to the appropriate processors, just as done with the canonical perfect forms.

This was sufficient for most of our cases. For some of the complex high-incidence cases we manually treated a few of the lower-rank faces separately in parallel, if they were too much to handle by a single processor. The saving bank system allowed for an easy way to load these back into the main computation.

When using $n$ processors we would typically leave one processor to handle requests to the saving bank. Other processors could then query this processor and submit or receive the dual description of complex cases.

Our implementation can run for a specific amount of time, after which the progress is saved, and from which one can resume the computation later.
We typically used $16$ to $256$ cores for the more complex cases of high incidence.

\section{Results}
\label{sec:results}
In this section we will consider the cost of our computations, some data resulting from our enumeration, and the corollaries resulting from those.
Many of the questions we try to answer are inspired by a collection of open problems posed by Martinet~\cite{Martinet2015}.
We will generally denote forms by $Q_{a[,b]}$, where $a$ indicates half the number of minimum vectors, and $b$ optionally indicates the specific lattice it is related to or the size of its automorphism group.

\subsection{Classification of perfect forms in dimension 9}
We completely finished the execution of Voronoi's algorithm and thereby found a complete classification of \numprint{2237251040} non-similar perfect forms in dimension $9$, leading to \cref{thm:perfect_and_extreme_forms}.
Among this complete list of perfect lattices the laminated lattice $\Lambdalat_9$~\cite{chaundy1946arithmetic,conway1982laminated} is the densest lattice. By Voronoi's theory, in particular~\cref{lem:extreme_eutactic}, we can conclude from this in~\cref{cor:lambda9_is_densest} that $\Lambdalat_9$ is the unique densest lattice packing in dimension $9$ and we have the Hermite constant $\gamma_9 = 2$.

We now consider some further properties of the classification.
In \cref{tab:minQ} we show the number of non-similar forms by their kissing number $|\Min(Q)|$.
From $|\Min(Q)| \geq 2 \cdot 63$ this matches the partial enumeration from~\cite[Table 1]{vanWoerdenMasterThesis}, so no new high-incidence forms were found. This is expected as the high-incidence forms are often highly connected and found early.
Just as in dimension $8$ the minimal kissing number $n=\frac{1}{2}d(d+1)$ for a $d$-dimensional perfect form is attained by most, more than half, of the perfect forms. Above this we roughly see an exponential decrease as the kissing number increases. As was already known by Watson~\cite{G1971}, the maximum kissing number in dimension $9$ is $2 \cdot 136$ attained by $\Lambdalat_9$.

In \cref{tab:autQ} in the Appendix we display the number of non-similar forms by the order $|\Aut(Q)|$ of their automorphism group. More than $96.7\%$ of them have a trivial automorphism group of order $2$, and more than $99.9\%$ of them have an order of at most $4$. The largest automorphism group has order $2^9 \cdot 9!$ and is attained by the root lattice $\Dlat_9$~\cite[p.~117]{conway2013sphere}.

In \cref{tab:scaleQ} in the Appendix we display the number of non-similar forms by their \emph{scale}. The scale of a form is the (minimal) value of the first-minimum $\lambda_1(Q)$ when the form $Q$ is scaled to be integral. We see that the scale is at most $168$. As a result the coefficients of $Q \in \S_{>0}^9(\ZZ)$ (when reduced, or for example after picking a basis of minimum vectors when possible) are reasonably small and can be stored quite efficiently.

To determine the number of extreme forms we computed for each perfect form their eutaxy following~\cref{lem:extreme_eutactic}.
Recall that a form $Q$ is eutactic if there exist $\lambda_x > 0$ for $\pm x \in \Min(Q)$ such that we have a eutaxy relation $Q^{-1} = \sum_{\pm x \in \Min(Q)} \lambda_x x x^t$.
We used exact arithmetic to compute $Q^{-1}$ and a linear program that maximizes $\lambda$ with the constraints $\lambda_x \geq \lambda$, to compute if such a relation exists. If $\lambda \approx 0$ we ran a linear program with exact arithmetic.
The results are shown in \cref{tab:eutaxy} in the Appendix.
There are \numprint{7338582} eutactic, and thus extreme non-similar perfect forms in dimension $9$. Among these there are $10$ which are strongly eutactic (all $\lambda_x > 0$ identical).
There are also \numprint{5843} perfect forms which are only semi-eutactic (some $\lambda_x = 0$). Among these there are $2$ where the remaining non-zero $\lambda_x$ are identical which we, in line with the naming convention, will call \emph{strongly semi-eutactic}. We do not know of an efficient way to determine if a form is strongly semi-eutactic so there might be more than these $2$.

\begin{table}
\centering
\caption{The number of non-similar forms by half their kissing number $\frac{1}{2}|\Min(Q)|$.}\label{tab:minQ}
\begin{tabular}{llllll}\toprule
$\frac{1}{2}|\Min(Q)|$ & \# & $\frac{1}{2}|\Min(Q)|$ & \# & $\frac{1}{2}|\Min(Q)|$ & \# \\ \midrule
$45$ & $1353947672$ & $61$ & $2244$ & $77$ & $1$ \\
$46$ & $471756975$ & $62$ & $1713$ & $78$ & $1$ \\
$47$ & $267588732$ & $63$ & $641$ & $79$ & $2$ \\
$48$ & $84473357$ & $64$ & $634$ & $80$ & $12$ \\
$49$ & $37278163$ & $65$ & $236$ & $81$ & $3$ \\
$50$ & $13324560$ & $66$ & $203$ & $82$ & $4$ \\
$51$ & $5299974$ & $67$ & $172$ & $84$ & $2$ \\
$52$ & $2009292$ & $68$ & $74$ & $85$ & $2$ \\
$53$ & $903943$ & $69$ & $44$ & $88$ & $1$ \\
$54$ & $366796$ & $70$ & $42$ & $90$ & $2$ \\
$55$ & $155182$ & $71$ & $26$ & $91$ & $1$ \\
$56$ & $78919$ & $72$ & $21$ & $99$ & $1$ \\
$57$ & $31113$ & $73$ & $7$ & $129$ & $1$ \\
$58$ & $17207$ & $74$ & $3$ & $136$ & $1$ \\
$59$ & $8231$ & $75$ & $4$ &  &  \\
$60$ & $4820$ & $76$ & $6$ &  &  \\ \bottomrule
\end{tabular}
\end{table}

\subsection{The computational cost and high-incidence cases}
We started Voronoi's algorithm with an existing list of \numprint{100000} perfect forms with $2 \cdot 45$ minimal vectors. This allowed to make full use of the parallelism from the start.
In about \numprint{100000} core hours on the (by now inactive) Lisa National Compute Cluster of SURF and the Curta cluster of MCIA between $2021$ and $2024$, we treated all low-incidence perfect forms with kissing number at most $2 \cdot 58$ with standard dual description computations.

What remained were the \numprint{19155} perfect forms of medium or high incidence, with up to restricted linear equivalence \numprint{7441} distinct polyhedral cones to compute the dual description of.
For these cases we used our parallel implementation of the Recursive Adjacency Decomposition Method as explained in~\cref{sec:the_dual_description}.
Depending on the case we ran this method with $1$ to $256$ processes.
The total cost for computing the dual description under symmetry for these cases was about $2$ million core hours on the Curta cluster, which is almost $3$ months of continuous computation on $1000$ cores.
Including all development, preliminary runs, and related computations, we estimate the total cost at about $3$ million core hours.
These computations mainly took place from $2023$ to $2024$.
Of those $2$ million core hours of computation for the dual description, about $1.5$ million core hours were spent on only $10$ high-incidence cases which are displayed \cref{tab:cost_of_dual_description_cases}.

Surprisingly, the highest-incidence case with incidence $136$ corresponding to the Laminated lattice $\Lambdalat_9$, was not the most computationally intensive.
We see in \cref{tab:cost_of_dual_description_cases} that this is mainly the case due to the many extra linear automorphisms of the cone, i.e., its automorphism group is a factor $128$ larger than the subgroup induced by $\Aut(Q_{136,\Lambdalat_9})$, as was already observed in~\cite{vanWoerdenMasterThesis}.
Furthermore, the perfect form with incidence $129$ does not even appear in this list as its dual description easily follows from that of $\Elat_8$, again see~\cite{vanWoerdenMasterThesis}.

The hardest case was a perfect form with the relatively low incidence of $76$, but with a small automorphism group and no additional restricted linear automorphisms. We computed part of the face lattice to find more linear automorphisms but we found none. This case required more than \numprint{350000} core hours and gave \numprint{1549616491} orbits.

\begin{table}
\centering
\caption{Cost of dual description cases with more than $50$k core hours. These cases account for $1.5$ million of the total amount of $2$ million core hours spent on dual description instances.}\label{tab:cost_of_dual_description_cases}
\footnotesize
\begin{tabular}{rrrrrr}\toprule
$\frac{1}{2}|\Min(Q)|$ & Core hours & $|\text{linaut}(P)|$ & rays (orbits) & $|\text{aut}(Q)|$ & neighbours (orbits) \\ \midrule
$136$ & \numprint{59277} & \numprint{660602880} & \numprint{64001686} & \numprint{10321920} & \numprint{1038153863} \\
$84$ & \numprint{75467} & \numprint{12288} & \numprint{171496157} & \numprint{384} & \numprint{1514557045} \\
$99$ & \numprint{84197} & \numprint{589824} & \numprint{137739671} & \numprint{18432} & \numprint{1842205495} \\
$90$ & \numprint{85349} & \numprint{73728} & \numprint{185824962} & \numprint{2304} & \numprint{2058568310} \\
$74$ & \numprint{95784} & \numprint{128} & \numprint{333146387} & \numprint{16} & \numprint{1257559244} \\
$80$ & \numprint{97118} & \numprint{7680} & \numprint{108828919} & \numprint{480} & \numprint{764775430} \\
$81$ & \numprint{181570} & \numprint{1296} & \numprint{254734260} & \numprint{2592} & \numprint{254734260} \\
$80$ & \numprint{219437} & \numprint{128} & \numprint{772745513} & \numprint{256} & \numprint{772745513} \\
$82$ & \numprint{245030} & \numprint{432} & \numprint{680747757} & \numprint{864} & \numprint{680747757} \\
$76$ & \numprint{355554} & \numprint{24} & \numprint{1549616491} & \numprint{48} & \numprint{1549616491} \\\bottomrule
\end{tabular}
\end{table}

\subsection{Perfect forms only connected via high-incidence cases}
Recall that most of the computational cost was spent exploring the neighbours of those perfect forms with a large kissing number.
Almost all perfect forms can be found without this part, because with a few exceptions all perfect forms are connected with our starting list via low-incidence cases.
Here we discuss the few exceptions to this, i.e., those perfect forms that were only connected via high-incidence forms.

\begin{figure}
\begin{tikzpicture}[scale=0.9]

  \node (Q72) at (-2,0) {$Q_{72,D_9}$};
  \node (QA9) at (-2,-2) {$Q_{45,A_9}$};

  \begin{scope}[shift={(-2,0)}]
  \draw[thick, dashed] (Q72) -- (-.5,.7);
  \draw[thick, dashed] (Q72) -- (-.25,.85);
  \draw[thick, dashed] (Q72) -- (0,.9);
  \draw[thick, dashed] (Q72) -- (.25,.85);
  \draw[thick, dashed] (Q72) -- (.5,.7); 
  \end{scope}

  \node (Q74) at (0,0) {$Q_{74,256}$};
  \node (Q1440) at (0,-2) {$Q_{45,1440}$};

  \draw[thick, dashed] (Q74) -- (-.5,.7);
  \draw[thick, dashed] (Q74) -- (-.25,.85);
  \draw[thick, dashed] (Q74) -- (0,.9);
  \draw[thick, dashed] (Q74) -- (.25,.85);
  \draw[thick, dashed] (Q74) -- (.5,.7);

  \begin{scope}[shift={(5,0)}]
  \node (Q2880) at (0,0) {$Q_{90,\kappalat_9}$};
  \node (Q7680) at (-3,-2) {$Q_{45,7680}$};
  \node (Q1536) at (-1,-2) {$Q_{45,1536}$};
  \node (Q2304) at (1,-2) {$Q_{45,2304}$};
  \node (Q11520) at (3,-2) {$Q_{45,11520}$};

  \node (Q161280) at (3,-4) {$Q_{45,161280}$};
  \node (QA95) at (1,-4) {$Q_{45,\Alat_9^5}$};

  \draw[thick, dashed] (Q2880) -- (-.5,.7);
  \draw[thick, dashed] (Q2880) -- (-.25,.85);
  \draw[thick, dashed] (Q2880) -- (0,.9);
  \draw[thick, dashed] (Q2880) -- (.25,.85);
  \draw[thick, dashed] (Q2880) -- (.5,.7);
  \end{scope}

  \begin{scope}[shift={(10,0)}]
  \node (Q136) at (0,0) {$Q_{136,\Lambdalat_9}$};
  \node (QA92) at (0,-2) {$Q_{45,\Alat_9^2}$};

  \draw[thick, dashed] (Q136) -- (-.5,.7);
  \draw[thick, dashed] (Q136) -- (-.25,.85);
  \draw[thick, dashed] (Q136) -- (0,.9);
  \draw[thick, dashed] (Q136) -- (.25,.85);
  \draw[thick, dashed] (Q136) -- (.5,.7);
  \end{scope}

  \draw[thick] (Q72) -- (QA9);

  \draw[thick] (Q74) -- (Q1440);
  \draw[thick] (Q2880) -- (Q11520);
  \draw[thick] (Q2880) -- (Q2304);
  \draw[thick] (Q2880) -- (Q7680);
  \draw[thick] (Q2880) -- (Q1536);

  \draw[thick] (Q136) -- (Q11520);
  \draw[thick] (Q136) -- (Q2304);
  \draw[thick] (Q136) -- (Q1536);
  \draw[thick] (Q136) -- (QA92);
  \draw[thick] (Q136) -- (Q161280);

  \draw[thick] (Q11520) -- (Q2304);
  \draw[thick] (Q2304) -- (Q1536);
  \draw[thick] (Q1536) -- (Q7680);

  \draw[thick] (Q11520) -- (Q161280);
  \draw[thick] (Q161280) -- (QA95);

\end{tikzpicture}
\caption{Part of Voronoi graph showing all perfect forms that are only connected via high-incidence perfect forms with kissing number at least $2 \cdot 61$. $Q_{a,b}$ indicates a perfect form with kissing number $2a$ and $b$ the size of its automorphism group or a label indicating a specific known lattice. Their properties are listed in~\cref{tab:special_forms}.}
\label{fig:highincidencegraph}
\end{figure}
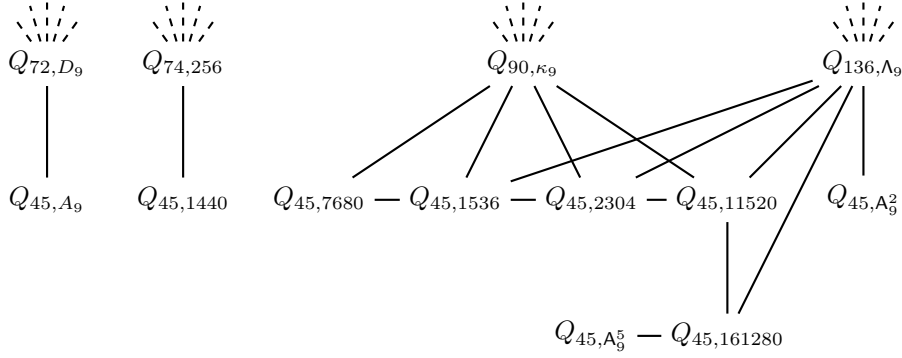

We consider all forms that were not found yet after incrementally exploring all forms until a kissing number of $2 \cdot 60$.
Concurrent to this work, such a computation was also performed by Gregory Minton~\cite{Minton2025} up to a kissing number of $2 \cdot 70$, during which he found \numprint{2237251031} perfect forms.

There are only four perfect forms with a large kissing number that lead to the remaining nine perfect forms, namely forms that we denote by $Q_{72,D_9}, Q_{74,256}, Q_{90,\kappalat_9}$ and $Q_{136,\Lambdalat_9}$ with a half-kissing number of respectively $72, 74, 90$ and $136$. The first one $Q_{72,D_9}$ and its singularly-connected neighbour $Q_{45,A_9}$ correspond to the $D_9$ and $A_9$ root lattices respectively. The latter two correspond to the lattices $\kappalat_9$~\cite{plesken1993constructing} and the laminated lattice $\Lambdalat_9$~\cite{chaundy1946arithmetic,conway1982laminated} respectively, both sections of the Leech lattice $\Lambda_{24}$~\cite{leech1967notes}. For our exploration the root lattice $A_9$ was already part of the starting set, so technically we did not find it at a later stage. 

The new perfect forms and their connections in Voronoi's graph are presented in~\cref{fig:highincidencegraph}.
Being only connected via these high-incidence perfect forms already makes them interesting, but in fact quite some of these lattices have other special properties.
All these new perfect forms have the minimal possible kissing number of $2 \cdot 45$, but all have a large automorphism group which is typically not the case for such low-incidence forms.
The perfect form $Q_{45,\Alat_9^2}$ is the unique perfect form that is only connected to the perfect lattice $\Lambdalat_9$, and corresponds to Coxeter's lattice $\Alat_9^2$~\cite{coxeter1951extreme}. We will see later that this lattice is the only $9$-dimensional perfect lattice whose minimal vectors do not span the lattice.
Then there is the perfect form $Q_{45,\Alat_9^5}$~\cite{coxeter1951extreme,martinet2004generalization}, dual to $Q_{45,\Alat_9^2}$, which is only connected in the Voronoi graph via two and three steps with $Q_{90,\kappalat_9}$ and $Q_{136,\Lambdalat_9}$ respectively. We will show later that this lattice is in fact the only so-called hollow perfect lattice in dimension $9$.
The perfect forms $Q_{45,\Alat_9^5}, Q_{45, \Alat_9^2}$ and $Q_{45,7680}$ all satisfy $\gamma'(Q)^2 = \lambda_1(Q) \cdot \lambda_1(Q^{-1}) = \frac{16}{5}$, the maximum known value for the Berg\'e-Martinet invariant $\gamma'(Q)$ in dimension $9$~\cite{berge1989probleme}. Along with the perfect form $Q_{45,1440}$, these $4$ forms are both extreme and dual-extreme, i.e., they achieve a local-maximum for the Berg\'e-Martinet invariant. Furthermore, the perfect form $Q_{45,161280}$ achieves a rather high value of $\gamma'(Q_{45,161280}) = 37/12$, but is not locally-optimal, due to its better neighbour $Q_{45,\Alat_9^5}$.

We were not able to link the perfect forms $Q_{45,1536}, Q_{45,2304}, Q_{45,11520}$ and $Q_{45,161280}$ to any previously known lattices in the literature.

\begin{table}
\caption{Some special perfect forms and their properties. Their relations in the Voronoi graph are shown in~\cref{fig:highincidencegraph}.}
\label{tab:special_forms}
\begin{tabular}{rrrrrrrr}\toprule
Form & $\frac{1}{2}|\Min(Q)|$ & $|\Aut{}|$ & $\lambda_1$ & $\lambda_1^{\text{dual}}$ & $\det(Q)$ & $\gamma(Q)$ & $\gamma'(Q)$ \\ \midrule
$Q_{45,A_9}$ & \numprint{45} & \numprint{7257600} & \numprint{2} & $9/10$ & $2 \cdot 5$ & 1.549 & $9/5$ \\\midrule
$Q_{72,D_9}$ & \numprint{72} & \numprint{185794560} & \numprint{2} & $1$ & $2^{2}$ & 1.714 & $2$ \\\midrule
$Q_{74,256}$ & \numprint{74} & \numprint{256} & \numprint{4} & $2/3$ & $2^{7} \cdot 3^{2}$ & 1.828 & $8/3$ \\\midrule
$Q_{90,\kappalat_9}$ & \numprint{90} & \numprint{2880} & \numprint{6} & $17/44$ & $2^{2} \cdot 3^{6} \cdot 11$ & 1.894 & $51/22$ \\\midrule
$Q_{136,\Lambdalat_9}$ & \numprint{136} & \numprint{10321920} & \numprint{4} & $1/2$ & $2^{9}$ & 2.0 & $2$ \\\midrule
$Q_{45,1440}$ & \numprint{45} & \numprint{1440} & \numprint{4} & $3/5$ & $2^{9} \cdot 5$ & 1.673 & $12/5$ \\\midrule
$Q_{45,\Alat_9^2}$ & \numprint{45} & \numprint{7257600} & \numprint{4} & $4/5$ & $2^{8} \cdot 5$ & 1.806 & $16/5$ \\\midrule
$Q_{45,11520}$ & \numprint{45} & \numprint{11520} & \numprint{12} & $1/4$ & $2^{6} \cdot 3^{3} \cdot 7^{5}$ & 1.778 & $3$ \\\midrule
$Q_{45,2304}$ & \numprint{45} & \numprint{2304} & \numprint{14} & $20/91$ & $2^{11} \cdot 7^{3} \cdot 13^{2}$ & 1.774 & $40/13$ \\\midrule
$Q_{45,7680}$ & \numprint{45} & \numprint{7680} & \numprint{6} & $8/15$ & $2 \cdot 3^{5} \cdot 5^{3}$ & 1.765 & $16/5$ \\\midrule
$Q_{45,1536}$ & \numprint{45} & \numprint{1536} & \numprint{16} & $11/56$ & $2^{17} \cdot 3^{2} \cdot 7^{3}$ & 1.769 & $22/7$ \\\midrule
$Q_{45,161280}$ & \numprint{45} & \numprint{161280} & \numprint{10} & $37/120$ & $2^{9} \cdot 3^{7} \cdot 5$ & 1.779 & $37/12$ \\\midrule
$Q_{45,\Alat_9^5}$ & \numprint{45} & \numprint{7257600} & \numprint{8} & $2/5$ & $2 \cdot 5^{8}$ & 1.771 & $16/5$ \\\bottomrule
\end{tabular}
\end{table}

\subsection{Possible kissing numbers in dimension $9$}

\begin{figure}
\includegraphics[width=0.9\textwidth]{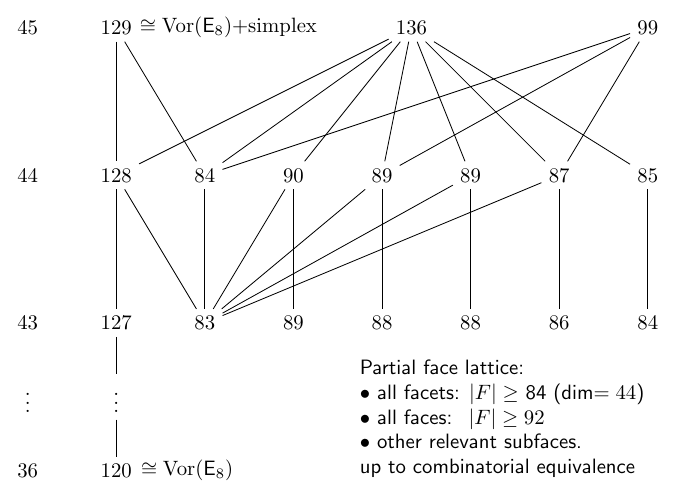}
\caption{Part of face (poset) lattice of $\Vor(Q_{99}), \Vor(Q_{129})$ and $\Vor(Q_{\Lambdalat_9})$, with combinatorially equivalent faces merged. The nodes are labelled with the incidence of the face. On the left the dimension of the face is denoted.}
\label{fig:face_lattice_high_incidence}
\end{figure}

As a by-product of our computation using the recursive dual description method we can also prove~\cref{thm:poss_kissing_numbers} which completely classifies the possible kissing numbers in dimension $9$.

\begin{proof}[{Proof of \cref{thm:poss_kissing_numbers}}]
Let $Q \in \S_{>0}^9$ be any $9$-dimensional PQF and $\Min(Q)$ its set of minimum vectors.
The claim of~\cref{thm:poss_kissing_numbers} is that $|\Min(Q)|$ lies in the set $2\cdot \{ 1, \ldots, 91, 99, 120, \ldots, 129, 136 \}$ and that all these values are attained.

By definition $Q$ lies on the face of the Rhyskov polyhedra $\P_{\lambda_1(Q)}^9$ defined by the facets corresponding to $\Min(Q)/\{ \pm \}$. The incidence of the face is $\frac{1}{2}|\Min(Q)|$ and thus corresponds to half the kissing number of $Q$. This face is either a vertex, i.e., $Q$ is a perfect form, or the face contains a vertex corresponding to a perfect form $Q'$. In the first case we know by our full classification that $|\Min(Q)| \in \{ 45, \ldots, 82, 84, 90, 91, 99, 129, 136 \}$. Furthermore, we know that $1, \ldots, 44$ are attained in dimension $8$~\cite{PerfectDim8} which by a direct-sum construction can be lifted to dimension $9$. In the second case the face can be identified with a non-trivial face of the tangent cone $\P(Q')$ at $Q'$.

To study the possible kissing numbers we thus have to investigate the face lattice of the tangent cones $\P(Q')$ for all perfect forms $Q' \in \S_{>0}^9$.
For example, incidences $83, \ldots 90$ are all attained by subfaces of $\P(Q_{\Lambdalat_9})$.
Because the incidence of a non-trivial face of $\P(Q')$ is always bounded by $\frac{1}{2}|\Min(Q')|-1$ we now only have to check the perfect forms with a kissing number of at least $93$, leaving only three cases: $Q_{99}, Q_{129}$ and $Q_{\Lambdalat_9}$.

For ease of explanation we switch to the equivalent dual setting and consider the faces of Voronoi domains.
The relevant parts of the face lattice of $\Vor(Q_{99}), \Vor(Q_{129})$ and $\Vor(Q_{\Lambdalat_9})$ are shown in \cref{fig:face_lattice_high_incidence}.
For $\Vor(Q_{99})$ our dual description computation shows that the its largest facet has incidence $89$, and thus there are no non-trivial faces with an incidence larger than that.

In~\cite{vanWoerdenMasterThesis} it was shown that $\Vor(Q_{129})$ can be decomposed into the Voronoi domain $\Vor(Q_{\Elat_8})$ of the $8$-dimensional root lattice $\Elat_8$ with $120$ extreme rays and a linearly independent simplicial cone with $9$ extreme rays. As a result, all faces of $\Vor(Q_{129})$ can be described as a sum of a face of $\Vor(Q_{\Elat_8})$ and a face of the simplicial cone. If we consider all faces containing $\Vor(Q_{\Elat_8})$ we thus obtain non-trivial faces of incidence $120+k$ for $k = 0, \ldots, 8$. The largest non-trivial face of $\Vor(Q_{\Elat_8})$ has incidence $75$, therefore all other faces of $\Vor(Q_{129})$ not containing $\Vor(Q_{\Elat_8})$ have incidence at most $75+9 = 84$.

Lastly, for $\Vor(Q_{\Lambdalat_9})$ we have one facet of incidence $128$, and the second largest facet has incidence $90$. We can thus focus on the subfaces of the largest facet. This facet is combinatorially isomorphic to the facet of $\P(Q_{129})$ with the same incidence. We have thus already treated this case.

We have thus shown that the only face incidence numbers occurring in the Rsyshkov Polyhedra are $\{1, \ldots,  91, 99, 120, \ldots, 129, 136 \}$ which concludes the proof.
\end{proof}

\subsection{Basis of minimal vectors}
Up to dimension $7$ all perfect lattices have a basis of minimum vectors, which was shown theoretically by Cs{\'o}ka~\cite{csoka1987there} and confirmed by the classifications of perfect forms in those dimensions.
Martinet~\cite[Conjecture 6.6.7]{martinet} conjectured that the same was true in dimension $8$, which was later confirmed by the classification~\cite{PerfectDim8}.
He also gave a counterexample for dimension $9$, namely Coxeter's lattice $\Alat_9^2$~\cite{coxeter1951extreme}.
Our classification shows that this is in fact the only case.
\begin{theorem}
  With the exception of Coxeter's lattice $\Alat_9^2$, all perfect lattices in dimension $9$ have a basis of minimum vectors.
\end{theorem}
Coxeter's lattice $\Alat_9^2$, is not even generated by its minimum vectors. Hence, every perfect lattice in dimension $9$ that is generated by its minimum vectors also has a basis of minimum vectors.

\subsection{Hollow perfect lattices}

A lattice of dimension $n$ is \emph{hollow}~\cite{martinet2004generalization} if it does not have any perfect $r$-dimensional sections with the same minimum for $1 < r < n$.
For $n \geq 10$ perfect hollow lattices with an odd minimum always exist~\cite{martinet2004generalization}. For $3 \leq n \leq 8$ (the scaling of) $\Elat_7^*$ is the unique perfect hollow lattice. This left open the case for dimension $9$.
\begin{theorem}
  All $9$-dimensional perfect lattices contain the perfect root lattice $\Alat_2$ as a section with the same first minimum, except Coxeter's lattice $\Alat_9^5$. Coxeter's lattice $\Alat_9^5$ is the unique hollow perfect lattice in dimension $9$. In particular, there are no hollow (in the sense of ~\cite{martinet2004generalization}) perfect lattices with an odd minimum in dimension $9$.
\end{theorem}
\begin{proof}
  Having a section with $\Alat_2$ with the same minimum is easily verified by computing if any pair of minimum vectors $x,y$ satisfies $\norm{x}^2 = \norm{y}^2 = 2 |\langle x,y \rangle|$.
  This is the case for all perfect lattices in dimension $9$ except for Coxeter's lattice $\Alat_9^5$, which therefore is the only remaining candidate for being hollow.
  % After checking this against our classification we remain with $\Alat_9^5$ as the only candidate for being hollow.
  Furthermore, $\Elat_7^*$ cannot occur as a section of $\Alat_9^5$ as their minima are $3$ and $8$ respectively (when scaled to an integral lattice). Because $\Elat_7^*$ is the only perfect lattice in dimension $2 \leq r < 9$ not containing $\Alat_2$ as a section with the same minimum we can conclude that $\Alat_9^5$ is hollow.
\end{proof}
Interestingly, Coxeter's lattice $\Alat_9^5$ is dual to the other remarkable perfect lattice $\Alat_9^2$, which is not generated by its minimum vectors. Furthermore, this matches the pattern that the perfect lattice $\Elat_7^* \cong (\Alat_7^2)^*$ is hollow.

\subsection{The Berg\'e-Martinet invariant}
The Bergé-Martinet invariant of a lattice $\lat$ is defined by $\gamma'(\lat) = \lambda_1(\lat) \cdot \lambda_1(\lat^*)$~\cite{berge1989probleme}.
One could interpret the Bergé-Martinet invariant as some sort of average packing density over the primal and dual lattice.
This invariant has a perfectness and eutaxy theory behind it similar to Hermite's invariant, which we will shortly recall. For more details see~\cite[Section 3.8]{martinet}.
Lattices that are a local maximum of $\gamma'$ are called \emph{dual-extreme}. Lattices represented by a form $Q$ are dual extreme if and only if they are dual-perfect ($\{ Qxx^tQ : x \in \Min(Q) \} \cup \{ yy^t : y \in \Min(Q^{-1}) \}$ span $\S^n$) and dual-eutactic (there exists a eutaxy relation $Q \left(\sum_{\pm x \in \Min(Q)} \lambda_x xx^t \right) Q = \sum_{\pm y \in \Min(Q^{-1})} \lambda_y yy^t$ where $\lambda_x, \lambda_y \in \RR_{>0}$ are strictly positive). Perfect lattices are trivially dual-perfect, and thus they are potentially a good source for dual-extreme forms.
\begin{theorem}
  There exist $18$ dual-extreme $9$-dimensional perfect non-similar lattices. All having $\gamma'(\lat)^2 \leq \frac{16}{5}$ with equality attained by $Q_{45,7680}$ and by the dual pair $(Q_{45,\Alat_9^2}, Q_{45,\Alat_9^5})$.
\end{theorem}
The three perfect cases attaining $\frac{16}{5}$ were already known by Martinet~\cite{martinet} and Baril~\cite{baril1996autour}.

\bibliographystyle{aomalpha}
\bibliography{LatticeRef}

@article{hales2005kepler,
  title={A proof of the Kepler conjecture},
  author={Hales, Thomas C},
  journal={Annals of mathematics},
  pages={1065--1185},
  year={2005},
  publisher={Princeton University}
}

@inproceedings{rigid_dim_six,
    author = {Dutour Sikiri\'c, M. and Vallentin, F.},
     title = {Some Six-Dimensional Rigid Forms},
 booktitle = {Voronoi's Impact on Modern Science, Book 3, edited by H. Syta, A. Yurachivsky, P. Engel},
    series = {Proc. Inst. Math. Nat. Acad. Sci. Ukraine},
    volume = {55},
     pages = {102--108},
     year  = {2005},
   address = {Kyiv 2005}
}

@inproceedings{hales2017formal,
  title={A formal proof of the Kepler conjecture},
  author={Hales, Thomas and Adams, Mark and Bauer, Gertrud and Dang, Tat Dat and Harrison, John and Le Truong, Hoang and Kaliszyk, Cezary and Magron, Victor and McLaughlin, Sean and Nguyen, Tat Thang},
  booktitle={Forum of mathematics, Pi},
  volume={5},
  pages={e2},
  year={2017},
  organization={Cambridge University Press}
}

@article{Opgenorth2001,
   author = {Opgenorth, Jürgen},
  journal = {Experimental Mathematics},
 language = {eng},
   number = {4},
    pages = {599-608},
publisher = {Taylor & Francis, Philadelphia},
    title = {Dual cones and the Voronoi algorithm.},
   volume = {10},
     year = {2001},
}

@article {BalinskiTh,
    AUTHOR = {Balinski, M. L.},
     TITLE = {On the graph structure of convex polyhedra in {$n$}-space},
   JOURNAL = {Pacific J. Math.},
  FJOURNAL = {Pacific Journal of Mathematics},
    VOLUME = {11},
      YEAR = {1961},
     PAGES = {431--434},
      ISSN = {0030-8730},
   MRCLASS = {52.10 (90.10)},
MRREVIEWER = {J. J. Stone},
}

@article{cohn2009leechoptimality,
  title={Optimality and uniqueness of the Leech lattice among lattices},
  author={Cohn, Henry and Kumar, Abhinav},
  journal={Annals of mathematics},
  pages={1003--1050},
  year={2009},
  publisher={JSTOR}
}

@mastersthesis{vanWoerdenMasterThesis,
  title={Perfect quadratic forms: an upper bound and challenges in enumeration},
  author={van Woerden, Wessel},
  year={2018},
  school={Leiden University}
}

@article{van2020upper,
  title={An upper bound on the number of perfect quadratic forms},
  author={van Woerden, Wessel PJ},
  journal={Advances in Mathematics},
  volume={365},
  pages={107031},
  year={2020},
  publisher={Elsevier}
}

@article{CanonicalFormPositiveForm,
  title={A canonical form for positive definite matrices},
  author={Dutour Sikiri{\'c}, Mathieu and Haensch, Anna and Voight, John and van Woerden, Wessel},
  journal={Open Book Series},
  volume={4},
  number={1},
  pages={179--195},
  year={2020},
  publisher={Mathematical Sciences Publishers}
}

@article{CUT8_facet,
    AUTHOR = {Deza, Michel and Dutour Sikiri{\'c}, Mathieu},
     title = {Enumeration of the facets of cut polytopes over some highly symmetric graphs},
   JOURNAL = {Intl. Trans. in Op. Res.},
  FJOURNAL = {International Transactions in Operational Research},
    volume = {23},
    number = {5},
     pages = {853--860},
      year = {2016},
 publisher = {Wiley Online Library},
}

@article{IsoEdgeSixDim,
    author = {Dutour Sikiri{\'c}, Mathieu and van Woerden, Wessel},
     title = {Complete classification of six-dimensional iso-edge domains},
   JOURNAL = {Acta Crystallogr. Sect. A},
  FJOURNAL = {Acta Crystallographica Section A. Foundations of
              Crystallography},
    volume = {81},
     pages = {9--15},
      year = 2024,
      ISSN = {0108-7673},
}

@article {AnzinEnumeration,
    AUTHOR = {Anzin, M.M. and Andreev, A.S.},
     TITLE = {The enumeration of $9$-dimensional perfect forms. Review of results (in Russian)},
   JOURNAL = {Mathematical aspects of modeling},
  FJOURNAL = {Mathematical aspects of modeling},
    VOLUME = {90},
      YEAR = {2013},
    NUMBER = {2},
     PAGES = {55--64},
}

@article {CR_decomposition_parallelization,
    AUTHOR = {Christof, Thomas and Reinelt, Gerhard},
     TITLE = {Decomposition and parallelization techniques for enumerating
              the facets of combinatorial polytopes},
   JOURNAL = {Internat. J. Comput. Geom. Appl.},
  FJOURNAL = {International Journal of Computational Geometry \&
              Applications},
    VOLUME = {11},
      YEAR = {2001},
    NUMBER = {4},
     PAGES = {423--437},
      ISSN = {0218-1959},
   MRCLASS = {52B55},
MRREVIEWER = {Johann Linhart},
}

@article {CR_small_polytopes,
    AUTHOR = {Christof, Thomas and Reinelt, Gerhard},
     TITLE = {Combinatorial optimization and small polytopes},
      NOTE = {With discussion},
   JOURNAL = {Top},
  FJOURNAL = {Top},
    VOLUME = {4},
      YEAR = {1996},
    NUMBER = {1},
     PAGES = {1--64},
      ISSN = {1134-5764},
   MRCLASS = {90C27 (52B12 90C08)},
MRREVIEWER = {Anna Rycerz},
}

@incollection {achill-enumerating,
    AUTHOR = {Sch{\"u}rmann, Achill},
     TITLE = {Enumerating perfect forms},
 BOOKTITLE = {Quadratic forms---algebra, arithmetic, and geometry},
    SERIES = {Contemp. Math.},
    VOLUME = {493},
     PAGES = {359--377},
 PUBLISHER = {Amer. Math. Soc.},
   ADDRESS = {Providence, RI},
      YEAR = {2009},
   MRCLASS = {11H55},
MRREVIEWER = {Peter M. Gruber},
}

@article {GroupPolytopeLMS,
    AUTHOR = {Bremner, David and Dutour Sikiri{\'c}, Mathieu and Pasechnik,
              Dmitrii V. and Rehn, Thomas and Sch{\"u}rmann, Achill},
     TITLE = {Computing symmetry groups of polyhedra},
   JOURNAL = {LMS J. Comput. Math.},
  FJOURNAL = {LMS Journal of Computation and Mathematics},
    VOLUME = {17},
      YEAR = {2014},
    NUMBER = {1},
     PAGES = {565--581},
      ISSN = {1461-1570},
   MRCLASS = {52B11 (20B25 52B15)},
MRREVIEWER = {Gabriel K. Cunningham},
}

@article {VoronoiI,
    AUTHOR = {Voronoi, George},
     TITLE = {{Nouvelles applications des param\`etres continues \`a la th\'eorie des formes quadratiques 1: Sur quelques propri\'et\'es des formes quadratiques positives parfaites}},
   JOURNAL = {J. Reine Angew. Math},
  FJOURNAL = {Journal für die Reine und Angewandte Mathematik},
    VOLUME = {133},
      YEAR = {1908},
    NUMBER = {1},
     PAGES = {97--178},
}

@incollection{montreal,
    AUTHOR = {Bremner, David and Dutour Sikiri\'{c}, Mathieu and Sch\"{u}rmann,
              Achill},
     TITLE = {Polyhedral representation conversion up to symmetries},
 BOOKTITLE = {Polyhedral computation},
    SERIES = {CRM Proc. Lecture Notes},
    VOLUME = {48},
     PAGES = {45--71},
 PUBLISHER = {Amer. Math. Soc., Providence, RI},
      YEAR = {2009},
   MRCLASS = {68W05 (52B15)},
}

@article {PleskenSouvignier,
    AUTHOR = {Plesken, Wilhelm and Souvignier, Bernd},
     TITLE = {Computing isometries of lattices},
      NOTE = {Computational algebra and number theory (London, 1993)},
   JOURNAL = {J. Symbolic Comput.},
  FJOURNAL = {Journal of Symbolic Computation},
    VOLUME = {24},
      YEAR = {1997},
    NUMBER = {3-4},
     PAGES = {327--334},
      ISSN = {0747-7171},
   MRCLASS = {11H56 (11Y16)},
MRREVIEWER = {Christine Bachoc},
}

@book {martinet,
    AUTHOR = {Martinet, Jacques},
     TITLE = {Perfect lattices in {E}uclidean spaces},
    SERIES = {Grundlehren der Mathematischen Wissenschaften [Fundamental
              Principles of Mathematical Sciences]},
    VOLUME = {327},
 PUBLISHER = {Springer-Verlag, Berlin},
      YEAR = {2003},
     PAGES = {xxii+523},
      ISBN = {3-540-44236-7},
   MRCLASS = {11H31 (11H06 11H55 11H56)},
MRREVIEWER = {Detlev W. Hoffmann},
}

@article {PerfectDim8,
    AUTHOR = {Dutour Sikiri{\'c}, Mathieu and Sch{\"u}rmann, Achill and
              Vallentin, Frank},
     TITLE = {Classification of eight-dimensional perfect forms},
   JOURNAL = {Electron. Res. Announc. Amer. Math. Soc.},
  FJOURNAL = {Electronic Research Announcements of the American Mathematical
              Society},
    VOLUME = {13},
      YEAR = {2007},
     PAGES = {21--32},
      ISSN = {1079-6762},
   MRCLASS = {11H55 (11H31)},
MRREVIEWER = {Gabriele Nebe},
}

@incollection {Leon1,
    AUTHOR = {Leon, Jeffrey S.},
     TITLE = {Permutation group algorithms based on partitions. {I}.
              {T}heory and algorithms},
      NOTE = {Computational group theory, Part 2},
   JOURNAL = {J. Symbolic Comput.},
  FJOURNAL = {Journal of Symbolic Computation},
    VOLUME = {12},
      YEAR = {1991},
    NUMBER = {4-5},
     PAGES = {533--583},
      ISSN = {0747-7171},
   MRCLASS = {68Q40 (20B40)},
MRREVIEWER = {Jean Moulin-Ollagnier},
}

@incollection {Leon2,
    AUTHOR = {Leon, Jeffrey S.},
     TITLE = {Partitions, refinements, and permutation group computation},
 BOOKTITLE = {Groups and computation, {II} ({N}ew {B}runswick, {NJ}, 1995)},
    SERIES = {DIMACS Ser. Discrete Math. Theoret. Comput. Sci.},
    VOLUME = {28},
     PAGES = {123--158},
 PUBLISHER = {Amer. Math. Soc., Providence, RI},
      YEAR = {1997},
   MRCLASS = {20B40 (68R05)},
MRREVIEWER = {J. D. Dixon},
}

@article {ContactLeech,
    AUTHOR = {Dutour Sikiri{\'c}, Mathieu and Sch{\"u}rmann, Achill and Vallentin, Frank},
     TITLE = {The contact polytope of the {L}eech lattice},
   JOURNAL = {Discrete Comput. Geom.},
  FJOURNAL = {Discrete \& Computational Geometry. An International Journal
              of Mathematics and Computer Science},
    VOLUME = {44},
      YEAR = {2010},
    NUMBER = {4},
     PAGES = {904--911},
      ISSN = {0179-5376},
     CODEN = {DCGEER},
   MRCLASS = {52C07},
MRREVIEWER = {Csaba Bir{\'o}},
}

@article {ComplexityVoronoiDSV,
    AUTHOR = {Dutour Sikiri{\'c}, Mathieu and Sch{\"u}rmann, Achill and
              Vallentin, Frank},
     TITLE = {Complexity and algorithms for computing {V}oronoi cells of
              lattices},
   JOURNAL = {Math. Comp.},
  FJOURNAL = {Mathematics of Computation},
    VOLUME = {78},
      YEAR = {2009},
    NUMBER = {267},
     PAGES = {1713--1731},
      ISSN = {0025-5718},
     CODEN = {MCMPAF},
   MRCLASS = {11H31 (11H56 52B12 52B55 68Q25 68U05)},
MRREVIEWER = {Gabriele Nebe},
}

@book {bookschurmann,
    AUTHOR = {Sch{\"u}rmann, Achill},
     TITLE = {Computational geometry of positive definite quadratic forms},
    SERIES = {University Lecture Series},
    VOLUME = {48},
      NOTE = {Polyhedral reduction theories, algorithms, and applications},
 PUBLISHER = {American Mathematical Society, Providence, RI},
      YEAR = {2009},
     PAGES = {xvi+162},
      ISBN = {978-0-8218-4735-0},
   MRCLASS = {11H55 (05B40 11J70 20B25 52-02 52B15 52B55)},
MRREVIEWER = {Peter M. Gruber},
}

@incollection {BirkhoffDualDesc,
    AUTHOR = {Dutour, Mathieu},
     TITLE = {The {B}irkhoff polytope of the groups ${\mathsf F}_4$ and ${\mathsf H}_4$},
 BOOKTITLE = {Proceedings of the 4th Croatian Combinatorial Days CroCoDays 2022},
     PAGES = {21--26},
      YEAR = {2022},
 publisher = {Civil Engineering faculty of Zagreb},
}

@manual{PolyhedralCpp,
    AUTHOR = {Dutour Sikiri{\'c}, Mathieu and van Woerden, Wessel},
     TITLE = {Polyhedral},
      YEAR = {2025},
       URL = {https://github.com/MathieuDutSik/polyhedral_common},
}

@manual{Polyhedral,
    AUTHOR = {Dutour Sikiri\'c, Mathieu},
     TITLE = {Polyhedral},
      YEAR = {2015},
       URL = {https://github.com/MathieuDutSik/GAPpackages},
}

@manual{MPIforum,
    key = {MPI25},
    AUTHOR = {MPI committee},
     TITLE = {MPI Forum},
      YEAR = {2025},
       URL = {https://www.mpi-forum.org/},
}

@manual{GAP,
    organization = "The GAP~Group",
    title        = "{GAP -- Groups, Algorithms, and Programming, Version 4.14.0}",
    year         = 2024,
    url          = "https://www.gap-system.org",
}

@manual{cdd,
    AUTHOR = {Fukuda, Komei},
     TITLE = {The cdd program},
      YEAR = {2015},
       URL = {http://www.ifor.math.ethz.ch/\~fukuda/cdd\_home/cdd.html},
}

@manual{permutalib,
    AUTHOR = {Dutour Sikiri\'c, Mathieu},
     TITLE = {The permutalib library},
      YEAR = {2024},
       URL = {https://github.com/MathieuDutSik/permutalib},
}

@manual{lrs,
    AUTHOR = {Avis, David},
     TITLE = {The lrs program},
      YEAR = {2015},
       URL = {http://cgm.cs.mcgill.ca/~avis/C/lrslib/USERGUIDE.html},
       NOTE = {we used version 7.2}
}

@manual{nauty,
    AUTHOR = {McKay, Brendan D. and Piperno, Alfredo},
     TITLE = {nauty and Traces},
       URL = {http://cs.anu.edu.au/people/bdm/nauty/},
       YEAR = {2022},
       NOTE = {we used version 2.8.6.}
}

@manual{bliss,
    AUTHOR = {Junttila, Tommi and Kaski, Petteri},
     TITLE = {bliss},
       URL = {http://www.tcs.hut.fi/Software/bliss/},
       YEAR = {2025},
}

@article{jefferson2019minimal,
  title={Minimal and canonical images},
  author={Jefferson, Christopher and Jonauskyte, Eliza and Pfeiffer, Markus and Waldecker, Rebecca},
  journal={Journal of Algebra},
  volume={521},
  pages={481--506},
  year={2019},
  publisher={Elsevier}
}

@article{gauss1831besprechung,
  title={Besprechung des Buchs von LA Seeber: Untersuchnugen iiber die Eigenschaften der positiven ternaren quadratischen Formen usw},
  author={Gauss, Carl Friedrich},
  journal={Gottingsche Gelehrte Anzeigen},
  year={1831}
}

@article{korkine1877formes,
  title={Sur les formes quadratiques positives},
  author={Korkine, Aleksandr and Zolotareff, G},
  journal={Mathematische Annalen},
  volume={11},
  number={2},
  pages={242--292},
  year={1877},
  publisher={Springer Berlin Heidelberg Berlin/Heidelberg}
}

@article{blichfeldt1935minimum,
  title={The minimum values of positive quadratic forms in six, seven and eight variables},
  author={Blichfeldt, Hans Frederick},
  journal={Mathematische Zeitschrift},
  volume={39},
  pages={1--15},
  year={1935},
  publisher={Springer}
}

@article{de1770demonstration,
  title={D{\'e}monstration d’un th{\'e}oreme d’arithm{\'e}tique},
  author={de Lagrange, Joseph Louis},
  journal={Nouveau M{\'e}moire de l’Acad{\'e}mie Royale des Sciences de Berlin},
  pages={123--133},
  year={1770}
}

@article{barnes1957perfect,
  title={The perfect and extreme senary forms},
  author={Barnes, ES},
  journal={Canadian Journal of Mathematics},
  volume={9},
  pages={235--242},
  year={1957},
  publisher={Cambridge University Press}
}

@inproceedings{jaquet1993enumeration,
  title={{\'E}num{\'e}ration complete des classes de formes parfaites en dimension 7},
  author={Jaquet-Chiffelle, David-Olivier},
  booktitle={Annales de l'institut Fourier},
  volume={43},
  number={1},
  pages={21--55},
  year={1993}
}

@inproceedings{ryshkov1970polyhedron,
  title={The polyhedron $\mu$(m) and certain extremal problems of the geometry of numbers},
  author={Ryshkov, Sergei Sergeevich},
  booktitle={Doklady akademii nauk},
  volume={194},
  number={3},
  pages={514--517},
  year={1970},
  organization={Russian Academy of Sciences}
}

@article{minkowski1905diskontinuitatsbereich,
  title={Diskontinuit{\"a}tsbereich f{\"u}r arithmetische {\"A}quivalenz.},
  author={Minkowski, Hermann},
  year={1905},
  publisher={Walter de Gruyter, Berlin/New York Berlin, New York}
}

@article{martinet2004generalization,
  title={A generalization of some lattices of {Coxeter}},
  author={Martinet, J and Berge, A-M},
  journal={Mathematika: A journal of pure and applied mathematics},
  number={101},
  pages={49--62},
  year={2004},
  publisher={Department of Mathematics}
}

@book{G1971,
author = {G. L. Watson},
language = {eng},
location = {Warszawa},
publisher = {Instytut Matematyczny Polskiej Akademi Nauk},
title = {The number of minimum points of a positive quadratic form},
url = {http://eudml.org/doc/268650},
year = {1971},
}

@article {BagnaraHZ08SCP,
  title = {The {Parma Polyhedra Library}: Toward a Complete Set of Numerical Abstractions for the Analysis and Verification of Hardware and Software Systems},
  journal = {Science of Computer Programming},
  volume = {72},
  number = {1{\textendash}2},
  year = {2008},
  pages = {3{\textendash}21},
  publisher = {Elsevier},
  issn = {0167-6423},
  doi = {10.1016/j.scico.2007.08.001},
  author = {Roberto Bagnara and Patricia M. Hill and Enea Zaffanella}
}

@article{bacher2018number,
  title={On the number of perfect lattices},
  author={Bacher, Roland},
  journal={Journal de th{\'e}orie des nombres de Bordeaux},
  volume={30},
  number={3},
  pages={917--945},
  year={2018}
}

@manual{flint,
  key = {{FLINT}},
  author = {The {FLINT} team},
  title = {{FLINT}: {F}ast {L}ibrary for {N}umber {T}heory},
  year = {2022},
  note = {we used version 2.9.0, \url{https://flintlib.org}}
}

@software{NetCDF_Software,
  author    = {{Unidata}},
  title     = {{Network Common Data Form (netCDF)} version 4.9.0},
  year      = {2022},
  publisher = {{Boulder, CO: UCAR/Unidata}},
  doi       = {10.5065/D6H70CW6},
}

@misc{Appleby_MurmurHash3,
  author       = {Appleby, Austin},
  title        = {{MurmurHash3}},
  howpublished = {Source code and algorithm description},
  year         = {2011},
  url          = {https://github.com/aappleby/smhasher}
}

@inproceedings{babai2016graph,
  title={Graph isomorphism in quasipolynomial time},
  author={Babai, L{\'a}szl{\'o}},
  booktitle={Proceedings of the forty-eighth annual ACM symposium on Theory of Computing},
  pages={684--697},
  year={2016}
}

@inproceedings{babai2019canonical,
  title={Canonical form for graphs in quasipolynomial time: preliminary report},
  author={Babai, L{\'a}szl{\'o}},
  booktitle={Proceedings of the 51st Annual ACM SIGACT Symposium on Theory of Computing},
  pages={1237--1246},
  year={2019}
}

@misc{mckay2025nauty,
  title={Nauty and traces user’s guide (version 2.9.0)},
  author={McKay, Brendan D and Piperno, Adolfo},
  year={2025},
  url={https://pallini.di.uniroma1.it/nug29.pdf}
}

@article{mckay2014practical,
  title={Practical graph isomorphism, {II}},
  author={McKay, Brendan D and Piperno, Adolfo},
  journal={Journal of symbolic computation},
  volume={60},
  pages={94--112},
  year={2014},
  publisher={Elsevier}
}

@phdthesis{van2023lattice,
  title={Lattice cryptography, from cryptanalysis to new foundations},
  author={van Woerden, Wessel Pieter Jacobus},
  year={2023},
  school={Ph. D. thesis, Fac. Sci., Math. Inst.(MI), Leiden Univ., Leiden, The Netherlands}
}

@article{viazovska2017sphere,
  title={The sphere packing problem in dimension 8},
  author={Viazovska, Maryna S},
  journal={Annals of mathematics},
  pages={991--1015},
  year={2017},
  publisher={JSTOR}
}

@article{cohn2017sphere,
  title={The sphere packing problem in dimension 24},
  author={Cohn, Henry and Kumar, Abhinav and Miller, Stephen and Radchenko, Danylo and Viazovska, Maryna},
  journal={Annals of mathematics},
  volume={185},
  number={3},
  pages={1017--1033},
  year={2017},
  publisher={Department of Mathematics, Princeton University Princeton, New Jersey, USA}
}

@MISC{eigenweb,
  author = {Ga\"{e}l Guennebaud and Beno\^{i}t Jacob and others},
  title = {Eigen v3},
  howpublished = {http://eigen.tuxfamily.org},
  year = {2010}
 }

@misc{boostweb,
  author = {{The Boost community}},
  title = {{B}oost {C++} {L}ibraries},
  howpublished = {\url{http://www.boost.org}},
  year = {2025}
}

@Manual{GMP3,
 title =   "{GNU MP}: {T}he {GNU} {M}ultiple {P}recision
      {A}rithmetic {L}ibrary",
 author = "Torbjörn Granlund and {the GMP development team}",
 edition =   "6.3.0",
 year =    2023,
 note = "\url{http://gmplib.org/}"
}

@article{wang1982p,
  title={P-adic reconstruction of rational numbers},
  author={Wang, Paul S and Guy, MJT and Davenport, James H},
  journal={ACM SIGSAM Bulletin},
  volume={16},
  number={2},
  pages={2--3},
  year={1982},
  publisher={ACM New York, NY, USA}
}

@book{grunbaum1967convex,
  title={Convex polytopes},
  author={Gr{\"u}nbaum, Branko},
  volume={16},
  year={1967},
  publisher={Springer}
}

@book{conway2013sphere,
  title={Sphere packings, lattices and groups},
  author={Conway, John Horton and Sloane, Neil James Alexander},
  volume={290},
  year={2013},
  publisher={Springer Science \& Business Media}
}

@article{coxeter1951extreme,
  title={Extreme forms},
  author={Coxeter, HAROLD SCOTT MACDONALD},
  journal={Canadian Journal of Mathematics},
  volume={3},
  pages={391--441},
  year={1951},
  publisher={Cambridge University Press}
}

@article{berge1989probleme,
  title={Sur un probleme de dualit{\'e} li{\'e} aux spheres en g{\'e}om{\'e}trie des nombres},
  author={Berg{\'e}, Anne-Marie and Martinet, Jacques},
  journal={Journal of Number Theory},
  volume={32},
  number={1},
  pages={14--42},
  year={1989},
  publisher={Elsevier}
}

@article{csoka1987there,
  title={There exists a basis of minimal vectors in every n $\leq$ 7 dimensional perfect lattice},
  author={Cs{\'o}ka, G},
  journal={Ann. Univ. Sc. Budapest Eotvos, Sect. Math},
  volume={30},
  pages={245--258},
  year={1987}
}

@article{plesken1993constructing,
  title={Constructing integral lattices with prescribed minimum. II},
  author={Plesken, Wilhelm and Pohst, Michael},
  journal={Mathematics of computation},
  volume={60},
  number={202},
  pages={817--825},
  year={1993}
}

@article{chaundy1946arithmetic,
  title={The arithmetic minima of positive quadratic forms (I)},
  author={Chaundy, TW},
  journal={The Quarterly Journal of Mathematics},
  number={1},
  pages={166--192},
  year={1946},
  publisher={Oxford University Press}
}

@article{conway1982laminated,
  title={Laminated lattices},
  author={Conway, John Horton and Sloane, Neil JA},
  journal={Annals of Mathematics},
  pages={593--620},
  year={1982},
  publisher={JSTOR}
}

@article{leech1967notes,
  title={Notes on sphere packings},
  author={Leech, John},
  journal={Canadian Journal of Mathematics},
  volume={19},
  pages={251--267},
  year={1967},
  publisher={Cambridge University Press}
}

@phdthesis{baril1996autour,
  title={Autour de l'algorithme de Vorono{\"\i}: construction de r{\'e}seaux euclidiens},
  author={Baril, Jean-Luc},
  year={1996},
  school={Bordeaux 1}
}

@misc{Martinet2015,
  author    = {Martinet, Jacques},
  title     = {Some Problems in the Theory of Euclidean Lattices},
  year      = {2015},
  month     = {January},
  howpublished = {\url{https://jamartin.perso.math.cnrs.fr/problat.pdf}},
  note      = {Last updated January 2015.}
}

@dataset{DutourSikiric2025Complete,
  author       = {Dutour Sikiric, Mathieu and van Woerden, Wessel},
  title        = {{Complete classification of perfect lattices in dimension 9}},
  publisher    = {Zenodo},
  year         = {2025},
  doi          = {10.5281/zenodo.15707640},
  url          = {https://doi.org/10.5281/zenodo.15707640}
}

@unpublished{Minton2025,
  author  = {Minton, Gregory},
  title   = {Partial enumeration of perfect forms in dimension 9},
  note    = {Personal correspondence},
  month   = {September},
  year    = {2025}
}

\appendix
\newpage
\section{}

\begin{table}[h]
\centering
\caption{The number of non-similar forms by the order $|\Aut(Q)|$ of their automorphism group.}\label{tab:autQ}
\begin{tabular}{llllllll} \toprule
$|\text{aut}(Q)|$ & \# & $|\text{aut}(Q)|$ & \# & $|\text{aut}(Q)|$ & \# & $|\text{aut}(Q)|$ & \# \\ \midrule
$2$ & $2165405948$ & $80$ & $10$ & $512$ & $1$ & $5760$ & $3$ \\
$4$ & $70134203$ & $84$ & $1$ & $576$ & $11$ & $6144$ & $1$ \\
$6$ & $495$ & $96$ & $564$ & $648$ & $1$ & $7680$ & $3$ \\
$8$ & $1411717$ & $108$ & $1$ & $672$ & $2$ & $8640$ & $1$ \\
$10$ & $4$ & $120$ & $3$ & $768$ & $12$ & $9216$ & $1$ \\
$12$ & $232514$ & $128$ & $34$ & $864$ & $6$ & $10080$ & $1$ \\
$16$ & $29670$ & $144$ & $105$ & $960$ & $7$ & $11520$ & $1$ \\
$18$ & $2$ & $160$ & $5$ & $1152$ & $3$ & $15552$ & $1$ \\
$20$ & $111$ & $168$ & $5$ & $1440$ & $6$ & $17280$ & $2$ \\
$24$ & $30189$ & $192$ & $90$ & $1536$ & $5$ & $18432$ & $1$ \\
$28$ & $6$ & $216$ & $1$ & $1728$ & $1$ & $36864$ & $1$ \\
$32$ & $2252$ & $240$ & $31$ & $1920$ & $2$ & $62208$ & $1$ \\
$36$ & $31$ & $256$ & $10$ & $2304$ & $9$ & $80640$ & $1$ \\
$40$ & $54$ & $288$ & $64$ & $2592$ & $3$ & $161280$ & $3$ \\
$48$ & $2268$ & $384$ & $26$ & $2880$ & $4$ & $725760$ & $1$ \\
$56$ & $9$ & $400$ & $1$ & $3072$ & $1$ & $7257600$ & $3$ \\
$64$ & $226$ & $432$ & $2$ & $3456$ & $2$ & $10321920$ & $1$ \\
$72$ & $257$ & $480$ & $25$ & $4608$ & $3$ & $185794560$ & $1$ \\ \bottomrule
\end{tabular}
\end{table}

\begin{table}
\centering
\caption{The number of non-similar forms by their scale. The scale of a PQF $Q \in \S_{>0}^d$ is the first-minimum $\lambda_1(s \cdot Q)$ with $s > 0$ minimal such that $s \cdot Q \in \S_{>0}^d$ is integral.}\label{tab:scaleQ}
\begin{tabular}{llllllll}\toprule
scale & \# & scale & \# & scale & \# & scale & \# \\ \midrule
$2$ & $2$ & $42$ & $1563418$ & $82$ & $946$ & $122$ & $8$ \\
$4$ & $998$ & $44$ & $1043502$ & $84$ & $1167$ & $124$ & $13$ \\
$6$ & $1001342$ & $46$ & $605352$ & $86$ & $483$ & $126$ & $6$ \\
$8$ & $25561692$ & $48$ & $445426$ & $88$ & $566$ & $128$ & $11$ \\
$10$ & $119744270$ & $50$ & $252229$ & $90$ & $349$ & $130$ & $3$ \\
$12$ & $253889032$ & $52$ & $187457$ & $92$ & $329$ & $132$ & $13$ \\
$14$ & $348860535$ & $54$ & $115379$ & $94$ & $182$ & $134$ & $3$ \\
$16$ & $370176078$ & $56$ & $86019$ & $96$ & $288$ & $136$ & $1$ \\
$18$ & $329241468$ & $58$ & $49839$ & $98$ & $111$ & $138$ & $3$ \\
$20$ & $261023035$ & $60$ & $43670$ & $100$ & $114$ & $140$ & $3$ \\
$22$ & $187497727$ & $62$ & $23639$ & $102$ & $79$ & $144$ & $5$ \\
$24$ & $127651488$ & $64$ & $20404$ & $104$ & $74$ & $146$ & $1$ \\
$26$ & $81448877$ & $66$ & $12779$ & $106$ & $37$ & $148$ & $2$ \\
$28$ & $51530332$ & $68$ & $10298$ & $108$ & $52$ & $150$ & $2$ \\
$30$ & $30986740$ & $70$ & $6061$ & $110$ & $37$ & $152$ & $1$ \\
$32$ & $19160686$ & $72$ & $6160$ & $112$ & $36$ & $156$ & $2$ \\
$34$ & $11188676$ & $74$ & $3132$ & $114$ & $30$ & $160$ & $2$ \\
$36$ & $7073814$ & $76$ & $2969$ & $116$ & $28$ & $168$ & $1$ \\
$38$ & $4085080$ & $78$ & $1928$ & $118$ & $8$ &  &  \\
$40$ & $2642786$ & $80$ & $1692$ & $120$ & $33$ &  &  \\ \bottomrule
\end{tabular}
\end{table}

\begin{table}
\centering
\caption{The number of non-similar perfect forms by their kissing number and eutaxy status. This list of strongly semi-eutactic perfect forms might be incomplete.}\label{tab:eutaxy}
\footnotesize
\scalebox{1.}[0.88]{
\begin{tabular}{llll}\toprule
$\frac{1}{2}|\Min(Q)|$ & \# perfect & \# semi-eutactic (strong) & \# eutactic (strong) \\ \midrule
$45$ & $1353947672$ & $289$ ($1$) & $986573$ ($5$) \\
$46$ & $471756975$ & $576$ & $1494396$ \\
$47$ & $267588732$ & $830$ & $1597977$ \\
$48$ & $84473357$ & $900$ ($1$) & $1264705$ \\
$49$ & $37278163$ & $891$ & $861126$ \\
$50$ & $13324560$ & $721$ & $522199$ \\
$51$ & $5299974$ & $563$ & $294577$ \\
$52$ & $2009292$ & $380$ & $154287$ \\
$53$ & $903943$ & $253$ & $78352$ \\
$54$ & $366796$ & $163$ & $39911$ \\
$55$ & $155182$ & $81$ & $20708$ \\
$56$ & $78919$ & $50$ & $10762$ \\
$57$ & $31113$ & $45$ & $5532$ \\
$58$ & $17207$ & $35$ & $3092$ \\
$59$ & $8231$ & $14$ & $1759$ \\
$60$ & $4820$ & $13$ & $941$ \\
$61$ & $2244$ & $11$ & $563$ \\
$62$ & $1713$ & $14$ & $378$ \\
$63$ & $641$ & $5$ & $246$ ($1$) \\
$64$ & $634$ & $2$ & $156$ \\
$65$ & $236$ & $2$ & $87$ \\
$66$ & $203$ & $2$ & $75$ \\
$67$ & $172$ & $0$ & $47$ \\
$68$ & $74$ & $0$ & $27$ \\
$69$ & $44$ & $1$ & $22$ \\
$70$ & $42$ & $1$ & $21$ \\
$71$ & $26$ & $0$ & $10$ \\
$72$ & $21$ & $0$ & $17$ ($2$) \\
$73$ & $7$ & $0$ & $4$ \\
$74$ & $3$ & $0$ & $3$ \\
$75$ & $4$ & $0$ & $4$ \\
$76$ & $6$ & $0$ & $3$ \\
$77$ & $1$ & $0$ & $1$ \\
$78$ & $1$ & $0$ & $1$ \\
$79$ & $2$ & $0$ & $1$ \\
$80$ & $12$ & $1$ & $5$ \\
$81$ & $3$ & $0$ & $3$ ($2$) \\
$82$ & $4$ & $0$ & $3$ \\
$84$ & $2$ & $0$ & $2$ \\
$85$ & $2$ & $0$ & $0$ \\
$88$ & $1$ & $0$ & $0$ \\
$90$ & $2$ & $0$ & $2$ \\
$91$ & $1$ & $0$ & $1$ \\
$99$ & $1$ & $0$ & $1$ \\
$129$ & $1$ & $0$ & $1$ \\
$136$ & $1$ & $0$ & $1$ \\ \midrule
Total & $2237251040$ & $5843$ ($2$) & $7338582$ ($10$) \\ \bottomrule
\end{tabular}
}
\end{table}

\end{document}